\documentclass[12pt]{article}

\usepackage{graphicx,psfrag,epsfig}
\usepackage{amssymb,amsmath,amscd,amsthm}
\usepackage{graphicx,psfrag,epsfig}

\usepackage{graphicx}

\newtheorem{theorem}{Theorem}[section]
\newtheorem{corollary}[theorem]{Corollary}

\newtheorem{lemma}[theorem]{Lemma}

\def\Const{{\rm Const}}

\def\eps{{\varepsilon}}

\def\brc{{\bar c}}

\def\brtau{{\bar\tau}}

\def\cO{\mathcal{O}}

\def\Tt{{\widetilde t}}

\begin{document}

\title{Motion in a Random Force Field}

\author{ Dmitry Dolgopyat, Leonid Koralov}

\date{}
\maketitle

\centerline{Department of Mathematics, University of Maryland}
\centerline{College Park, MD 20742} \centerline{
dmitry@math.umd.edu, koralov@math.umd.edu}

\begin{abstract}
 We consider the motion of a particle in a random isotropic
  force field.
 Assuming that the force field arises from a Poisson field in $\mathbb{R}^d$, $d \geq 4$, and the
 initial velocity of the particle is sufficiently large, we
 describe the asymptotic behavior of the particle.
\medskip
%\noindent
%{\bf Key words}:
\end{abstract}


\section {Introduction}
\label{se1} Let $F$ be a random force field on $\mathbb{R}^d$
defined on a probability space $(\Omega', \mathcal{F}',
\mathrm{P}')$.  The motion of a particle is described by the
equation
\begin{equation} \label{large}
\ddot{X}(t) = F(X(t)),
\end{equation}
where $X(t)$ denotes the position of the particle at time $t$. Let
$V(t) = \dot{X}(t)$ be the velocity of the particle at time $t$.
As initial conditions we take $X(0) = 0$ and $V(0) = v_0 $, where
$v_0$ is a non-random vector. The force field is assumed to be
stationary and isotropic. The precise form of the force field will
be discussed below.

We shall be interested in the asymptotic behavior of $X(t)$ and
$V(t)$ as $t \rightarrow \infty$. The process $V(t)$ can be
written in the integral form as
\begin{equation} \label{vel}
V(t) = v_0 +  \int_0^t F(X(s)) d s.
\end{equation}
Formal arguments, based on the near-independence of contributions
to the integral on the right-hand side of (\ref{vel}) from
non-intersecting sub-intervals, suggest that $V(t)$ behaves as a
diffusion process, if time is re-scaled appropriately. In fact, we
shall prove that there is an event $\Omega'_{v_0}$ in the
underlying probability space $\Omega'$, such that
$\mathrm{P}'(\Omega' \setminus \Omega'_{v_0}) \rightarrow 0$ as
$|v_0| \rightarrow \infty$, and $ V(c ^3  t) /c $ converges, as $t
\rightarrow \infty$, to a diffusion process on $\Omega'_{v_0}$
(the probability measure on $\Omega'_{v_0}$ is defined by
conditioning  $\mathrm{P}'$ on the event $\Omega'_{v_0}$). In
particular, the kinetic energy of the particle will be shown to
tend to infinity as $t \rightarrow \infty$. The precise
formulation of these results will be provided in Section~\ref{mr}.

We cannot, however, expect  that  $ V(c ^3  t) /c $ converges to
the diffusion process for almost all realizations of the force
field, if $v_0$ is fixed. Indeed, depending on the assumptions
imposed on $F$, the trajectory may remain in a bounded region of
space and the velocity may remain bounded with positive
probability.

It must be noted that we must exclude the case  $F =\nabla H$,
where $H$ is a stationary field, since in this case $(X(t), V(t))$
is a Hamiltonian flow with the Hamiltonian $ \overline{H}(k,x) =
|k|^2/2 - H(x)$, and $|V(t)|^2/2 - H(X(t))$ is constant on the
solutions of (\ref{large}).

Earlier papers primarily studied the behavior of $X(t)$ and $V(t)$
on long time intervals, whose length, however, depended on
$|v_0|$, where the initial velocity $v_0$ was treated as a large
parameter. We shall assume that $v_0$ is fixed and $t$ tends to
infinity. The trade-off is that we need to exclude an event of
small but positive measure from the underlying probability space.

Let us mention some of the earlier results concerning the
long-time behavior of $X(t)$ and $V(t)$. In \cite{KP}, Kesten and
Papanicolaou considered the equation
\begin{equation} \label{small}
\ddot{x}(t) = \varepsilon F(x(t))
\end{equation}
with the initial data $x(0) = 0$ and $v(0) = \widetilde{v}$.
Certain mixing assumptions were imposed on the force field $F$. It
was shown that if $d \geq 3$, the process $v(t/\varepsilon^2)$
converges weakly to a diffusion process $\overline{v}(t)$ with the
initial data $\overline{v}(0) = \widetilde{v}$. The generator of
the limiting process can be written out explicitly. The process
$\varepsilon^2 x(t/\varepsilon^2)$ converges weakly to
$\overline{x}(t) = \int_0^t \overline{v}(s) d s$.

Note that equations (\ref{large}) and (\ref{small}) are related
via the change of variables
\begin{equation}
\label{V-F}
X(t) = x (t /\sqrt{\varepsilon}),~~v_0 =
\widetilde{v}/\sqrt{\varepsilon}.
\end{equation}
Therefore, the convergence result for $v(t/\varepsilon^2)$ can be
formulated in terms of $V(t)$ as follows: the process $ V(|v_0|^3
t) / |v_0| $ converges to a diffusion process when $v_0$ tends to
infinity in such a way that $v_0 /|v_0| = \widetilde{v}$ remains
fixed. Similarly, $X(|v_0|^3 t)/ |v_0|^4$ converges weakly to a
limiting process.

In \cite{DGL}, Durr, Goldstein, and Lebowitz extended the
convergence results to the two-dimensional case. The field $F$ was
assumed to be a gradient of $H(x) = \sum_i h(x - p_i)$, where $h$
is a smooth function with compact support, and the points $p_i$
form a Poisson field on the plane. An additional difficulty in the
two-dimensional case is that, unlike the case with $d \geq 3$,
typical trajectories of (\ref{large}) will self-intersect. In
\cite{KR2}, Komorowski and Ryzhik proved the two-dimensional
result in the case when $H$ is sufficiently mixing, but is not
necessarily generated by a Poisson field.

In \cite{KR}, Komorowski and Ryzhik considered the process
(\ref{large}) on a longer time scale. Namely, they demonstrated
that $X(|v_0|^{3+ 8 \alpha} t)/ |v_0|^{4(1 +\alpha)}$ converges to
a Brownian motion for all sufficiently small $\alpha > 0$. It was
assumed that $F = \nabla H$, where $H$ is sufficiently mixing.

Unlike the above papers, we shall consider the asymptotic behavior
of $V(t)$ when $v_0$ is  fixed and $t \rightarrow \infty$. First,
however, assume that $v_0 \rightarrow \infty$, $v_0/|v_0| =
\widetilde{v}$, and let $\overline{V}(t)$ be the limiting process
for $ V(|v_0|^3 t) / |v_0| $ as $v_0 \rightarrow \infty$. (It
satisfies the stochastic differential equation (\ref{diff1}),
below.) As has been noted by Dolgopyat and De La Llave in
\cite{DD}, the process $\overline{V}(t)$ is self-similar, that is
for $c > 0$ the process $ \overline{V}(c^3 t) /c$ satisfies the
same stochastic differential equation with initial condition $
\overline{V}(0) /c$. Therefore, for $c$ fixed, $ V(c ^3 |v_0|^3 t)
/(c  |v_0|) $ tends to the diffusion process (\ref{diff1})
starting at $\widetilde{v}/c$. If, instead, we assume that $v_0$
is large but fixed, and take the limit as $c \rightarrow \infty$,
we formally obtain that $ V(c ^3 t) /c $ tends to the diffusion
process (\ref{diff1}) starting at the origin. We remark that the diffusion
processes satisfying the self-similarity property described above are well
understood (see e.g. \cite{RY}, Section XI). In particular the fact that
this process is non-recurrent for $d>3$ plays a crucial role in our analysis.

\section{The Force Field}
\label{forfie}

Let $S_{R, m}$ be the space of smooth functions $f: \mathbb{R}^d
\rightarrow \mathbb{R}^d$ which are supported inside the ball or
radius $R$ centered at the origin and satisfy $||f||_{C^2(
\mathbb{R}^d)} \leq m$. Let
 $\mu$ be a probability measure on  $S_{R, m}$. We assume that
 $\mu$ is symmetric in the sense that if $\psi: S_{R, m}
 \rightarrow S_{R, m}$ is a mapping that maps a function $f$ into
 $-f$, then
 \begin{equation} \label{symmetr}
 \mu(U) = \mu(\psi(U))
 \end{equation}
 for any measurable set $U \subseteq S_{R, m}$.
%We assume that
%\[
%\int_{S_{R, m}}  \int_{ \mathbb{R}^d} f(x) d x d \mu (f) = 0.
%\]
 We also assume that $\mu$ is isotropic, that is the vectors
 $(O^{-1} f ( Ox_1),...,O^{-1} f(Ox_n))$   and $(f(x_1),...,f(x_n))$ have the
same distribution for each orthogonal matrix $O$ and points
$x_1,...,x_n \in \mathbb{R}^d$. Suppose that on a probability
space $(\Omega', \mathcal{F}', \mathrm{P}')$ we have a sequence of
functions $f_i : \Omega' \rightarrow S_{R, m}$ which are
independent and identically distributed with distribution~$\mu$.
We shall consider random vector fields $F$ of the form
\begin{equation} \label{ff}
F(x) = \sum_{i =1}^\infty f_i(x - r_i),
\end{equation}
where
 $r_i$ form a Poisson point field with unit intensity on $
\mathbb{R}^d$. We assume that the Poisson field
% is defined over the same probability space $(\Omega', \mathcal{F}', \mathrm{P}')$ and
is independent of the sequence $f_i$. Note that the force
field~$F$ defined by (\ref{ff}) is stationary, isotropic, and has
zero mean. We shall denote the $j$-th coordinate of the vector~$F$
by~$F^j$.

\section{Formulation of the Main Result}
\label{mr}

Let $W_t$ be a standard $d$-dimensional Brownian motion. Consider
the $d$-dimensional process $\overline{V}(t)$ which satisfies the
diffusion equation
\begin{equation} \label{diff1}
d \overline{V}(t) =  \frac{1}{\sqrt{|\overline{V}(t)|}} \left(
\lambda d W_t + (\sigma -\lambda
)\frac{\overline{V}(t)}{|\overline{V}^2(t)|} (\overline{V}(t), d
W_t)\right) + \frac{((d-2)\sigma^2-(d-1)\lambda^2)
\overline{V}(t)}{ 2 |\overline{V}(t)|^3} d t,
\end{equation}
where
\begin{equation} \label{sig}
\sigma^2 = \int_{-\infty}^\infty \mathrm{E} (F^1 (0) F^1(e_1 t)) d
t,~~\lambda^2 = \int_{-\infty}^\infty \mathrm{E} (F^2 (0) F^2(e_1
t)) d t,
\end{equation}
and $e_1$ is the first coordinate vector. It is clear that the
integrals defining $\sigma^2$ and $\lambda^2$ are non-negative. We
shall require that
\begin{equation} \label{csi}
\int_{-\infty}^\infty \mathrm{E} (F^1 (0) F^1(e_1 t)) d t
> 0.
\end{equation}
Thus, the case when $F = \nabla H$, where $H$ is a stationary
random field, is excluded from consideration.
 The generator of the process $\overline{V}(t)$ is
\[
L = \frac{1}{2} \sum_{i,j =1}^d \frac{\partial}{\partial v_i}
a_{ij}(v) \frac{\partial}{\partial v_j},
\]
where
\[
a_{ij}(v) = \int_{-\infty}^\infty \mathrm{E} (F^i (0) F^j(v t)) d
t.
\]
By examining the stochastic differential equation satisfied by
$|\overline{V}(t)|^2/2$ (see formula (\ref{energy}) below), it is
it follows that
the origin is an inaccessible point for the
process $\overline{V}(t)$ if $d \geq 3$ (see \cite{RY}, Section XI).
Therefore the solution of
(\ref{diff1}) with initial condition $\overline{V}(0) \neq 0$
exists for all $t$. By the solution with the initial condition
$\overline{V}(0) = 0$ we shall mean the limit in distribution, as
$\overline{V}(0) \rightarrow 0$, of solutions with initial
condition $\overline{V}(0)$. We shall prove the following theorem.

\begin{theorem} \label{t1}
Let $F$ be a  vector field in $\mathbb{R}^d$, $d \geq 4$, given by
(\ref{ff}), which satisfies (\ref{csi}). For each sufficiently
large $v_0$ there is a set $\Omega'_{v_0}$ such that $\lim_{|v_0|
\rightarrow \infty} \mathrm{P}'(\Omega'_{v_0})=1$ and if
$\Omega'_{v_0}$ is viewed as a probability space with the measure
obtained by conditioning $P'$ on the event $\Omega'_{v_0}$, i.e.
${\mathrm{P}}_{v_0}'(A) =
{\mathrm{P}'(A)}/{\mathrm{P}'(\Omega'_{v_0})}$, then

(a)  the processes $X(t)$ and $V(t)$ tend to infinity almost
surely,

(b) the processes $V(c^3t)/c$ on $\Omega'_{v_0}$ converge in
distribution, as $c \rightarrow \infty$, to the solution
of~(\ref{diff1}) with the initial condition $\overline{V}(0) = 0$.
\end{theorem}
Let $E(t) = |V(t)|^2/2$ be the kinetic energy of the particle at
time $t$, and $\overline{E}(t) = |\overline{V}(t)|^2/2$, where
$\overline{V}(t)$ is the solution of (\ref{diff1}) with initial
condition $\overline{V}(0) = 0$. By the Ito formula,
$\overline{E}(t)$ is the solution of
\begin{equation} \label{energy}
d \overline{E}(t) = \sigma (2 \overline{E}(t))^{1/4} d B_t +
\frac{\sigma^2 (d-1)}{2 \sqrt{2} \overline{E}(t)} d t
\end{equation}
with the initial condition $\overline{E}(0) = 0$, where $B_t$ is a
standard one-dimensional Brownian motion.  Let
\begin{equation} \label{position}
\overline{X}(t) = \int_0^t \overline{V}(s) d s.
\end{equation}

We observe that the fact that our force is Poisson and rotation
invariant (rather than a general strongly mixing force) is
primarily used in subsection \ref{SSSI}. An alternative approach
would be to estimate the rate of convergence in the averaging
theorem (our Lemma \ref{ole}) using the techniques of \cite{DK} or
\cite{KR} but this would make the proof much more complicated.
Therefore in this paper to we consider the simplest possible force
distribution leaving the extension to more general force fields as
an open question.

\noindent {\bf Remark.} Up to a change of time by a constant
factor, $\overline{E}^{3/4}(t)$, and consequently
$|\overline{V}(t)|^{3/2}$, are Bessel processes with dimension
$2d/3$. Therefore if $d>3$, then from the properties of the Bessel
processes (see Chapter 3.3.C of \cite{KS}) it follows that $\ln
|\overline{V}(t)|$ is a diffusion process with a positive drift,
and therefore
\begin{equation}
\label{BiasUp} \mathrm{P}\left(|\overline{V}(t)| \text{ reaches }
2|v_0| \text{ before } |v_0|/2\right)
>\frac{1}{2}.
\end{equation}
 Moreover, $~\lim_{t \rightarrow \infty} |\overline{V}(t)|
= \infty~$ almost surely for $~d  >3$. We shall also see that
$~{\lim_{t \rightarrow \infty} |{V}(t)| = \infty}$ with high
probability with respect to the measure $ \mathrm{P}'$ if the
initial velocity is large (see Lemma~\ref{ffi}). These properties
will allow us to conclude that with high probability the
trajectories of $\overline{X}(t)$ and $X(t)$ do not ``come close''
to self-intersecting if the initial velocity is large (see
Lemma~\ref{nntt}). This avoidance of near self-intersections is
essential to the proof of Theorem~\ref{t1}.

% In $d = 3$ the situation is more delicate since the process $|\overline{V}(t)|$
% is recurrent. We believe that part (b) of Theorem~\ref{t1} still
% holds, and the proof can be modified by considering the intervals
% of time when $|\overline{V}(t)|$ is small separately. In $d =2$
% the trajectories of $\overline{X}(t)$ and $X(t)$ will
% self-intersect, and the result, if true, will require a much more
% technical proof.

Theorem~\ref{t1} immediately implies the following.
\begin{corollary} \label{t2}
Let $F$ be a  vector field in $\mathbb{R}^d$, $d \geq 4$, given by
(\ref{ff}), which satisfies (\ref{csi}). For each sufficiently
large $v_0$ there is a set $\Omega'_{v_0}$ such that $\lim_{|v_0|
\rightarrow \infty} \mathrm{P}'(\Omega'_{v_0}) = 1 $ and if
$\Omega'_{v_0}$ is viewed as a probability space with the measure
obtained by conditioning $P'$ on the event $\Omega'_{v_0}$, then

(a) the processes $E(c^3t)/c^2$ on $\Omega'_{v_0}$ converge in
distribution, as $c \rightarrow \infty$, to the solution
of~(\ref{energy}) with the initial condition $\overline{E}(0) =
0$. The processes $X(c^3 t)/c^4$ on $\Omega'_{v_0}$ converge in
distribution, as $c \rightarrow \infty$, to the process
$\overline{X}(t)$ defined by (\ref{position}).

(b) There exists a constant $\brc$ such that
$E(t)/\brc t^{2/3}$ converges in distribution to a random variable
with density
 \[
p(x)=\frac{3}{2 \Gamma(d/3)} x^{\frac{d}{2} - 1}
\exp\left({-x^{\frac{3}{2}}}\right).
\]
\end{corollary}

If $d=2$ or $3$ then the situation is more delicate since
$\overline{V}$ is recurrent. it seems that the methods of the
present paper can be modified to show that
$P(|\dot{Y}(t)|\to\infty)=0$ (in the more difficult case $d=2$
where the trajectories of $X(t)$ self sintersect the methods of
\cite{DGL, KR2} should be used). We also beleive that Theorem
\ref{t1}(b) remains valid for $d=2, 3$ (the theorem is false if
$d=1$ since in that case all orbits are periodic). However we do
not have a proof of this since the cases of large $V$ and small
$V$ need to be considered separately and new ideas are necessary
to handle the latter case.

\section{Auxiliary Processes}
\subsection{Time Discretization} \label{tdis}
 Let $X(t)$ be the solution of  (\ref{large}) with
initial conditions $X(0) = 0$, $V(0) = v_0$. We assume that the
field $F$ and, consequently, the process $X(t)$  are  defined on a
probability space $(\Omega', \mathcal{F}', \mathrm{P}')$. Assume,
momentarily, that the trajectories of $X(t)$ always ``keep
exploring" new regions of $ \mathbb{R}^d$ in the sense that for
each $t \geq 0$ the tail  of the trajectory $X(s), s \geq t+1$, is
separated from the initial part of the trajectory $X(s), s \leq
t$, by a distance larger than $2R$. Then, for large $t$, the
interval $[0,t]$ can be split into sub-intervals, such that the
contribution to the integral on the right-hand side of (\ref{vel})
from different sub-intervals are almost independent. This fact
will be helpful when proving that $V(c^3 t) /c$ converges to a
diffusion process.

We shall demonstrate that with high probability the trajectories
of the process $X(t)$ indeed have the desired property  if the
initial velocity is large.  To this end, we shall construct an
auxiliary process $Y(t)$ on a probability space $({\Omega},
{\mathcal{F}}, {\mathrm{P}})$. The process $Y(t)$ is defined as
the solution of
\begin{equation} \label{ypro}
\ddot{Y}(t) = \widetilde{F}(t, Y(t)),~~Y(0) = 0,~\dot{Y}(0) = v_0,
\end{equation}
where $ \widetilde{F}(t, x)$ can be obtained from $F(x)$ by
``switching on" new  independent versions of $F(x)$ at stopping
times $\tau_n$, as described below. Since the force field
$\widetilde{F}$ on the right-hand side of (\ref{ypro}) is
time-dependent, the increments $\dot{Y}(\tau_n) -
\dot{Y}(\tau_{n-1})$ and $\dot{Y}(\tau_k) - \dot{Y}(\tau_{k-1})$
will be almost independent if $|n - k|$ is large. This way, we
don't need to be concerned about possible self-intersections of
the process $Y(t)$ when studying the long-time behavior of the
process $\dot{Y}(t)$. Moreover, the introduction of the stopping
times $\tau_n$ will allow us to use a kind of Markov property: the
distribution of $ Y({\tau_n + \cdot}) - Y(\tau_n)$ will depend on
the events prior to $\tau_n$ only through $\dot{Y}({\tau_n})$ (see
Section~\ref{zpr} below).

On the other hand, we shall prove that the processes $X(t)$ and
$Y(t)$ will have  the same distribution if certain events with
small probabilities are excluded from their respective probability
spaces. More precisely, there are events $\Omega'_{v_0} \subseteq
\Omega'$ and ${\Omega}_{v_0} \subseteq {\Omega}$  such that the
processes $X(t,\omega') \chi_{{\Omega}'_{v_0}}(\omega')$ and
$Y(t,\omega) \chi_{{\Omega}_{v_0}}(\omega)$ have the same
distributions. The probabilities of $\Omega'_{v_0}$ and
${\Omega}_{v_0}$ tend to one when $|v_0| \rightarrow \infty$.

 Below we give a rigorous definition of the field
$\widetilde{F}(t, x)$. Roughly speaking, we follow the trajectory
$X(t)$ till time $\tau_1$ such that there are no points $r_i$, $i
\geq 1$, in the $2R$-neighborhood of $X(\tau_1)$. Then we replace
the force field $F$ by an independent version, also generated by
Poisson points with unit intensity on $ \mathbb{R}^d \setminus
B_{2R}(X(\tau_1))$, but zero intensity on $ B_{2R}(X(\tau_1))$,
where $B_{2R}(X(\tau_1))$ is the ball of radius $2R$ centered at
the $X(\tau_1)$. We can then treat $X(\tau_1)$ as the new initial
point and define the following stopping times by induction. More
precisely, let $i, n \geq 1$, and  $f^n_i$   be independent
identically distributed functions with distribution~$\mu$. Let
$F_0 = F$. Define the sequence of random fields
 $F_1,F_2,...$  as follows:
\[
F_n(x) = \sum_{i =1}^\infty f^n_i(x - r^n_i),
\]
where,  for each $n \geq 1$,  $r^n_i$ form a Poisson point field
with unit intensity on $ \mathbb{R}^d \setminus B_{2R}(0)$ and
zero intensity on $ B_{2R}(0)$, and $B_{2R}(0)$ is the ball of
radius $2R$ centered at the origin. The Poisson fields $r^0 , r^1,
r^2,...$ are assumed to be independent of each other and of
$f^n_i$ (here $r^0 = r$). We can assume that the random fields
$F_n$ are defined on a probability space $({\Omega},
{\mathcal{F}}, {\mathrm{P}})$, which is an extension of the
original probability space $(\Omega', \mathcal{F}', \mathrm{P}')$.
%\begin{remark}
%The distribution  of the field $F_0$ is slightly different from
%$F_1$, $F_2$,...
%\end{remark}
%
%
%
 Let $\tau_0
= 0$, $ \widetilde{F}_0 = F_0$, $Y_0 = 0$, and $v_0$ be the
initial condition for the process~$X(t)$. Assuming that
$\tau_{n-1}$, $ \widetilde{F}_{n-1}$, $Y_{n-1}$, and $v_{n-1}$
have been defined for some $n \geq 1$, we inductively define
$\tau_{n}$, $ \widetilde{F}_{n}$,  $Y_{n}$, and $v_n$. Let $y(t)$
be the solution of the equation
\[
\ddot{y}(t) = \widetilde{F}_{n-1}( y(t)),~~t \geq \tau_{n-1}
\]
with the initial conditions  $y(\tau_{n-1}) = Y_{n-1}$, $
\dot{y}(\tau_{n-1}) = v_{n-1}$. Let $l = 4R+1$. Let $\tau_n$ be
the first time after $\tau_{n-1}+ l |v_{n-1}|^{-1}$ when there are
no points $r^{n-1}_i$, $i \geq 1$, within the $2R$-neighborhood of
$y(t) - Y_{n-1}$, that is
\[
\tau_n = \inf\{t \geq \tau_{n-1}+ l |v_{n-1}|^{-1}: \inf_{i \geq
1}
 |y(t)- Y_{n-1} -r^{n-1}_i| \geq 2R \}.
\]
%Let $ \tau'_n $  be the first time after $\tau_{n-1}+
%|v_{n-1}|^{\alpha}$ when  the velocity $\dot{y}(t)$ is different
%from the initial velocity $v_{n-1}$ by at least $|v_{n-1}|^\beta$,
%where $\beta$  satisfies $(\alpha -1)/2 < \beta < 0$ (the precise
%value of $\beta$ is unimportant). That is,
%\[
%\tau'_n = \inf\{t \geq \tau_{n-1}+|v_{n-1}|^{\alpha}:
%|\dot{y}(t) - v_{n-1}| \geq |v_{n-1}|^\beta \}.
%\]
%Let
%\[
%\overline{\tau'}_n = \tau_{n-1} + 2|v_{n-1}|^{\alpha}.
% +q|v_{n-1}|^3.
%\]
%
%
% We define $\tau_n = \min{(\sigma_n,\tau'_n,
% \overline{\tau'}_n )}$.
If $\tau_n = \infty$, then $Y_i$, $v_i$ and $
\widetilde{F}_{i}(x)$ are undefined for $i \geq n$. Otherwise,
define $Y_n = y(\tau_n)$, $v_n = \dot{y}(\tau_n)$, and $
\widetilde{F}_{n}(x) = F_n(x - Y_n)$.

Now we can set $\widetilde{F}(t, x) = \widetilde{F}_{n-1}(x)$ for
$\tau_{n-1} \leq t < \tau_n$. Then the solution $Y(t)$
of~(\ref{ypro}) satisfies $Y(\tau_n) = Y_n$ and $\dot{Y}(\tau_n) =
v_n$.  The relation of $Y(t)$ to the original process $X(t)$ is
explained by the following lemma.
\begin{lemma} \label{lh1}
Let $Y(t)$ be the solution of (\ref{ypro}) on the probability
space $({\Omega}, {\mathcal{F}}, {\mathrm{P}})$. For each
sufficiently large $v_0$ there are events $\Omega'_{v_0} \subseteq
\Omega'$ and ${\Omega}_{v_0} \subseteq {\Omega}$ with the
following properties:

(a)  $\lim_{|v_0| \rightarrow \infty} \mathrm{P}'(\Omega'_{v_0}) =
\lim_{|v_0| \rightarrow \infty} {\mathrm{P}}({\Omega}_{v_0}) =1$.

(b) The processes $X(t)$ and $Y(t)$ have the same distribution if
restricted to the spaces $\Omega'_{v_0}$ and $
 {\Omega}_{v_0} $,
respectively.

 (c) The processes $\dot{Y}(t)$ and $Y(t)$ tend to
infinity almost surely on ${\Omega}_{v_0}$.

(d) If ${\Omega}_{v_0}$ is viewed as a probability space with the
measure obtained by conditioning $P$ on the event $\Omega_{v_0}$,
then the processes $\dot{Y}(c^3t)/c$ on ${\Omega}_{v_0}$ converge
in distribution, as $c \rightarrow \infty$, to the solution
of~(\ref{diff1}) with the initial condition $\overline{V}(0) = 0$.
\end{lemma}
It is clear that Theorem~\ref{t1} follows from Lemma~\ref{lh1}.
%\begin{lemma} \label{lh2}
%For any $\delta
%> 0$ there is $r$,
%such that for any $v_0$ with $|v_0| \geq r$ there is a set
%${\Omega}_{v_0}$ with the following properties:
%
%(a)  ${\mathrm{P}}({\Omega}_{v_0}) \geq 1 - \delta$.
%
%(b) If ${\Omega}_{v_0}$ is viewed as a probability space with the
%measure ${{\mathrm{P}}}_{v_0}(A) =
%{{\mathrm{P}}(A)}/{{\mathrm{P}}({\Omega}_{v_0})}$, then the
%processes $\dot{Y}(c^3t)/c$ on ${\Omega}_{v_0}$ converge in
%distribution, as $c \rightarrow \infty$, to the solution
%of~(\ref{diff1}) with the initial condition $\overline{V}(0) = 0$.
%\end{lemma}
 We shall study some of the properties of  $Y(t)$ in
Section~\ref{papr} and prove parts (a), (b) and (c) of
Lemma~\ref{lh1} in Section~\ref{ltbe}. We then prove part (d) of
Lemma~\ref{lh1} in Section~\ref{YY}.
\subsection{Another Auxiliary Process} \label{zpr}
Note that the distribution  of vector field $F_0$ is slightly
different from the distribution of the fields $F_n$, $n \geq 1$.
Namely, $F_0$ is a based on a Poisson field on  $ \mathbb{R}^d$,
while $F_n$, $n \geq 1$, are based on Poisson fields on  $
\mathbb{R}^d \setminus B_{2R}(0)$.

Consider the vector field $\overline{{F}}$, which is defined in
the same way as $ \widetilde{F}$, except that now we assume $F_0$
to be defined by a Poisson field with unit intensity on  $
\mathbb{R}^d \setminus B_{2R}(0)$ and zero intensity on $
B_{2R}(0)$. The process $Z(t)$ is defined as the solution of
\[
\ddot{Z}(t) = \overline{{F}}(t, Z(t)),~~Z(0) = 0,~\dot{Z}(0) =
w_0,
\]
where $w_0$ is a random vector independent of $\overline{{F}}$.
The reason to consider $Z(t)$ is  the following Markov property.

Let $ \mathcal{G}_n$ be the $\sigma$-algebra generated by $
\widetilde{F}_{i}$, $i \leq n -1$. For each $n \geq 1$, $A \in
\mathcal{B}(\mathbb{R}^d)$, and $B \in \mathcal{B}(C([0,
\infty)))$ we have:
\begin{equation} \label{smp}
 \mathrm{P}( Y({\tau_n +
\cdot}) - Y(\tau_n) \in B | \mathcal{G}_{n} )
\chi_{\{\dot{Y}({\tau_n}) \in A\}} =  \mathrm{P} (Z(\cdot) \in B)
\chi_{\{ w_0 \in A \}}~~{\rm in}~{\rm distribution,}
\end{equation}
 where the initial
velocity vector $w_0$ for the process $Z(t)$ is assumed to be
distributed as~$\dot{Y}(\tau_n)$.

If a random variable $\tau$ is such that $\tau(\omega) \in \{
\tau_1(\omega),\tau_2(\omega),...\}$ for each $\omega$, and the
set $\{\tau \leq \tau_n\}$ is $\mathcal{G}_n$-measurable for each
$n$, then from (\ref{smp}) it follows that
\begin{equation} \label{smp2}
 \mathrm{P}( Y({\tau +
\cdot}) - Y(\tau) \in B | \mathcal{G} ) \chi_{\{\dot{Y}({\tau})
\in A\}} =  \mathrm{P} (Z(\cdot) \in B) \chi_{\{ w_0 \in A
\}}~~{\rm in}~{\rm distribution,}
\end{equation}
where $\mathcal{G} = \{A \in \mathcal{F}: A \cap \{\tau \leq
\tau_n\} \in \mathcal{G}_n~~{\rm for}~{\rm each}~n\}$.

\section{Preliminaries}
In this section we recall some results about diffusion approximation for
the process $\dot{Y}(t)$ and provide bounds on probabilities of
some unlikely events.

%Properties of the Auxiliary Process $Y(t)$}
\label{papr}
\subsection{Behavior of $Y(t)$ and $\dot{Y}(t)$  on the Time Interval $[\tau_n, \tau_{n+1}]$}
\label{SSST}
In this subsection
%we shall approximate  the processes $Y(t)$ and $
%\dot{Y}(t)$ on the interval $[\tau_n,\tau_{n+1}]$ by simpler
%processes. These approximations will hold with high probability on
%the event where $|v_n|$ is large. In particular,
we shall prove that with high probability the velocity vector does
not change significantly between the times $\tau_n$ and
$\tau_{n+1}$ if $|v_n|$ is large. Therefore $Y(t)$ can be well
approximated by a straight line on this time interval.

Let $z_n(t) = Y_n +  (t - \tau_n)v_n $, that is $z_n(t)$ is the
solution of
\[
\ddot{z}_n(t) = 0,~~z_n(\tau_n) = Y_n,~\dot{z}_n(\tau_n) = v_n.
\]
%Let $ \tau'_n
%$, $n \geq 1$, be the first time after $\tau_{n-1}$ when the
%velocity $\dot{y}(t)$ is different from the initial velocity
%$v_{n-1}$ by at least $|v_{n-1}|^\beta$. That is,
%\[
%\tau'_n = \inf\{t \geq \tau_{n-1}: |\dot{y}(t) - v_{n-1}| \geq
%|v_{n-1}|^\beta \}.
%\]
%Let
%\[
%{\tau_n''} = \tau_{n-1} + |v_{n-1}|^{\alpha}.
%\]
Let $\eta_n$, $n \geq 1$ be the first time after $\tau_{n-1} + l
|v_{n-1}|^{-1}$ when there are no points $r^{n-1}_i$, $i \geq 1$,
within the $2R$-neighborhood of $z_{n-1}(t) - Y_{n-1}$, that is
\[
\eta_n = \inf\{t \geq \tau_{n-1}+ l|v_{n-1}|^{-1}: \inf_{i \geq 1}
 |z_{n-1}(t)- Y_{n-1} -r^{n-1}_i| \geq 2R \}.
\]
Let
\[
%f(t) = y(t) - z(t),~~~
\xi_n(t) = \int_{\tau_n}^{t} \widetilde{F}_{n}(z_n(s)) d
s,~~~\zeta_n(t) = \int_{\tau_n}^t \xi_n(s) d s,~~t \geq \tau_n.
\]
Let us first examine the behavior of $Y(t)$ on the interval $[0,
\tau_1]$.
\begin{lemma} \label{ftt01}
For each $N$ and $\delta > 0$ we have
\begin{equation} \label{aa1}
{\mathrm{P}}\left( \tau_1 > |v_0|^{-1+\delta}\right) \leq
|v_0|^{-N},
\end{equation}
\begin{equation} \label{aa2}
{\mathrm{P}}\left( \sup_{0 \leq t \leq \tau_1} |\dot{Y}(t) - v_0|
> |v_0|^{-1 + \delta}\right) \leq |v_0|^{-N},
\end{equation}
\begin{equation} \label{uu1}
{\mathrm{P}}\left(  \sup_{0 \leq t \leq \tau_1 } |\dot{Y}(t) - v_0
-\xi_0(t)| > |v_0|^{- 3 + \delta} \right) \leq |v_0|^{-N},
\end{equation}
\begin{equation} \label{uu2}
{\mathrm{P}}\left( |\xi_0(\tau_1) - \int_0^{\eta_1}
{F}_{0}(z_0(s)) d s | > |v_0|^{-1 + \delta} \right) \leq
|v_0|^{-N},
\end{equation}
for all sufficiently large $|v_0|$.
\end{lemma}
\proof Let $y(t)$, $t \geq 0$, be the solution of the equation
\begin{equation} \label{yyya}
\ddot{y}(t) = {F}_{0}(y(t)),~~y(0) = 0,~\dot{y}(0) = v_0,
\end{equation}
 Note that $Y(t)$ satisfies this equation on
the interval $[0,\tau_{1})$.

We shall say that an event (which depends on $v_0$) happens with
high probability if for each $N$ the probability of the complement
does not exceed $|v_0|^{-N}$ for all  sufficiently large $|v_0|$.
  Let us show
that for each $\delta
> 0$
\begin{equation} \label{pp1aa}
||{F}_{0}||_{C^2(B_{|v_0|}(0))} \leq \delta \ln |v_0|
\end{equation}
with high probability.  Recall the definition of ${F}_{0} = F$
from Section~\ref{forfie}, and note that
$||{F}_{0}||_{C^2(B_{|v_0|}(0))}$ may be larger than $\delta \ln
|v_0|$ only if there is a point $x \in B_{|v_0|}(0)$ such that the
ball of radius $R$ centered at~$x$ contains at least $[\delta \ln
|v_0|]/m$ points out of $r_i$, $i \geq 1$. The probability of this
event is easily seen to decay faster than any power of $|v_0|$
since $\delta$, $m$ and $R$ are constants and $r_i$, $i \geq 1$,
form a Poisson field with unit intensity.

Take $-1 < \alpha < 0$, which will be specified later, and  let
${T_0} = |v_0|^\alpha$.
 Let us show   that for each $\delta
> 0$
\begin{equation} \label{pp2a11}
\sup_{0 \leq t \leq T_0} |\xi_0(t)| \leq |v_0|^{(\alpha -1)/2 +
\delta}
%\widetilde{\mathrm{P}} ( \sup_{0 \leq t \leq 2} \xi(t) \geq
%|v_0|^{\delta - \frac{1}{2}} ) \leq |v_0|^{-N},
\end{equation}
with high probability. Indeed, let
\[
\overline{\xi}^k = \xi_0(\frac{3 (k+1) R}{|v_0|}) - \xi_0(\frac{3
k R}{|v_0|}),~~k \geq 0,
\]
where $R$ was defined in Section~\ref{forfie}. Note that
$\{\overline{\xi}^{2k} \}_{k \geq 0}$ and $\{\overline{\xi}^{2k+1}
\}_{k \geq 0}$ are sequences of independent identically
distributed random variables since the force field is uncorrelated
at distances larger than $2R$. Their tails decay faster than
exponentially since $F$ is based on a Poisson random field.
Therefore, the moderate deviation bounds (see, for example,
Theorem 9.4 of \cite{Bill}) imply that
\[
\sup_{m \leq |v_0|^{\alpha+1}/3R} (|\sum_{k =0}^m
\overline{\xi}^{2k} | + |\sum_{k =0}^m \overline{\xi}^{2k+1}|)
\leq |v_0|^{-1} |v_0|^{\frac{\alpha + 1}{2}  + \delta}
\]
since the standard deviation of $\overline{\xi}^{k}$ is of order
$1/|v_0|$. This easily implies (\ref{pp2a11}) since ${T_0 =
|v_0|^\alpha}$.

From (\ref{pp2a11}) it immediately follows that for each $\delta
> 0$
\begin{equation} \label{pp211}
\sup_{0 \leq t \leq T_0} |\zeta_0(t)| \leq |v_0|^{(3\alpha -1)/2 +
\delta}
% \widetilde{\mathrm{P}} ( \sup_{0 \leq t \leq 2} \zeta(t)
%\geq |v_0|^{\delta - \frac{1}{2}} ) \leq |v_0|^{-N}
\end{equation}
with high probability.  Let $\sigma = \inf\{t: y(t) \notin
{B_{|v_0|}(0)} \}$ (with the convention that the infimum of the
empty set is $+\infty$). By (\ref{yyya}),
\[
y(t) = z_0(t) + \zeta_0(t) + \int_{0}^t \int_{0}^u (F_0(y(s)) -
F_0(z_0(s))) d s d u.
\]
 Therefore,
\[
|y(t)- z_0(t)| \leq |\zeta_0(t)| + {T_0} ||F_0||_{ C^1
(B_{|v_0|}(0))} \int_{0}^t |y(s)-z_0(s)| d s~~~{\rm for}~~0 \leq t
\leq  T_0 \wedge \sigma.
\]
By (\ref{pp1aa}) and (\ref{pp211}), this implies that for each
$\delta
> 0$,
\[
\sup_{0 \leq t \leq T_0 \wedge \sigma} |y(t) - z_0(t)| \leq
|v_0|^{(3\alpha -1)/2 + \delta} +  \delta |v_0|^{2 \alpha}  \ln
|v_0| \sup_{0 \leq t \leq T_0 \wedge \sigma} |y(t) - z_0(t)|
\]
with high probability. Since  $\delta |v_0|^{2 \alpha}  \ln |v_0|
< 1/2$ for large enough $|v_0|$, this implies that
\[
\sup_{0 \leq t \leq T_0 \wedge \sigma} |y(t) - z_0(t)| \leq 2
|v_0|^{(3\alpha -1)/2 + \delta}
\]
with high probability and, consequently, $\sigma > T_0$ with high
probability.

 Since $\delta$ was an
arbitrary positive number, we obtain that for each $\delta > 0$
\begin{equation} \label{gr1aa}
\sup_{0 \leq t \leq T_0} |y(t) - z_0(t)| \leq |v_0|^{(3\alpha
-1)/2 + \delta}
\end{equation}
with high probability. By (\ref{yyya}),
\[
\dot{y}(t) - v_0 = \xi_0(t) + \int_{0}^t (F_0(y(s)) - F_0(z_0(s)))
d s.
\]
Due to (\ref{pp2a11}) and (\ref{gr1aa}),  for each $\delta
>0$
\begin{equation} \label{gr2aa}
\sup_{0 \leq t \leq T_0} |\dot{y}(t) - v_0| \leq |v_0|^{(\alpha
-1)/2 + \delta}
\end{equation}
with high probability.
 By  the expression $\langle \nabla F_0, v
\rangle$, where $v$ is a vector, we shall mean the vector $w$ with
components $w^j = \sum_{i = 1}^d {F_0^j}_{x_i} v^i$. If $y(t) \in
B_{|v_0|}(0)$ for $0 \leq t \leq T_0$, then by the Taylor formula
\[
 \sup_{0 \leq t \leq T_0} |\dot{y}(t) - v_0 -\xi_0(t)| \leq
\]
\[
\leq  \sup_{0 \leq t \leq T_0} |\int_{0}^t \langle \nabla
F_0(z_0(s)),(y(s) -z_0(s)) \rangle d s| +  \sup_{0 \leq t \leq
T_0} \frac{1}{2} \int_{0}^t ||F_0||_{ C^2 ( B_{|v_0|}(0))} |y(s)
-z_0(s)|^2 d s.
\]
From (\ref{pp1aa}) and (\ref{gr1aa}) it follows that for each
$\delta>0$ the second term in the right-hand side does not exceed
$|v_0|^{4 \alpha - 1 + \delta}$ with high probability. To estimate
the first term we use the fact that
\[
 \sup_{0 \leq t \leq T_0} |\int_{0}^t  {F_0^j}_{x_i}(z_0(s)) d s | \leq  |v_0|^{(\alpha
-1)/2  + \delta}~,~~1 \leq i, j \leq d,
\]
with high probability. Then, after integrating by parts and using
(\ref{gr1aa}) and (\ref{gr2aa}), we obtain that the first term in
the right-hand side does not exceed $|v_0|^{2 \alpha - 1 +
\delta}$ with high probability. Therefore, for each $\delta
>0$
\begin{equation} \label{r20aa}
 \sup_{0 \leq t \leq T_0} |\dot{y}(t) - v_0 -\xi_0(t)| \leq |v_0|^{2 \alpha - 1 +
\delta}
\end{equation}
with high probability.

Note that the set $\{z_0(t), t \in [T_0/4, T_0/2]\}$ is a straight
segment of length $|v_0|^{1+\alpha}/4$, and the points $r^0_i$, $i
\geq 1$, form a Poisson field. This implies that with high
probability there is a moment of time $t \in [T_0/4, T_0/2]$ such
that there are no points $r^0_i$, $i \geq 1$, within the
$4R$-neighborhood of $z_0(t)$. Therefore,
\begin{equation} \label{teyyx1}
\eta_{1}  \leq  T_0
\end{equation}
with high probability. Moreover, from the
 proximity of $y(t)$ and $z_0(t)$
(formula (\ref{gr1aa})) and the definition of $\tau_1$ it  now
follows
\begin{equation} \label{teyy}
\tau_{1}  \leq  T_0
\end{equation}
with high probability. Since $\alpha \in (0,1)$ was arbitrary,
this implies (\ref{aa1}). Combining (\ref{teyy}) with
(\ref{gr2aa}) and (\ref{r20aa}), we obtain (\ref{aa2}) and
(\ref{uu1}), respectively. Combining (\ref{teyyx1}) and
(\ref{teyy}) with~(\ref{pp2a11}), we obtain that for arbitrary
$\delta > 0$ we have
\[
|\int_0^{\eta_1} {F}_{0}(z_0(s)) d s| +  |\xi_0(\tau_1)| \leq
|v_0|^{-1 + \delta}
\]
with high probability, which implies (\ref{uu2}). \qed
\\

\noindent
 {\bf Remark.} Obviously, the same result holds if the process $Y(t)$ is
replaced by the process $Z(t)$ with initial velocity $v_0$.
Therefore, we have the following.
\begin{corollary} \label{fttt}
For each $N$ and $\delta > 0$ there is $r > 0$ such that for each
$n$
\begin{equation} \label{bb1}
{\mathrm{P}}\left( \tau_{n+1} - \tau_n > |v_n|^{-1+\delta} |
\mathcal{G}_{n} \right) \leq |v_n|^{-N},
\end{equation}
\begin{equation} \label{bb2}
{\mathrm{P}}\left( \sup_{\tau_n \leq t \leq \tau_{n+1}}
|\dot{Y}(t) - v_n|
> |v_n|^{-1 + \delta} |
\mathcal{G}_{n}\right) \leq |v_n|^{-N},
\end{equation}
\begin{equation} \label{bb3}
{\mathrm{P}}\left(  \sup_{\tau_n \leq t \leq \tau_{n+1} }
|\dot{Y}(t) - v_n -\xi_n(t)| > |v_n|^{- 3 + \delta} |
\mathcal{G}_{n}\right) \leq |v_n|^{-N},
\end{equation}
\begin{equation} \label{bb4}
{\mathrm{P}}\left( |\xi_n(\tau_{n+1}) - \int_{\tau_n}^{\eta_{n+1}}
\widetilde{{F}}_{n}(z_n(s)) d s | > |v_n|^{-1 + \delta} |
\mathcal{G}_{n}\right) \leq |v_n|^{-N},
\end{equation}
hold almost surely on the event $|v_n| > r$.
\end{corollary}
Let $H_n$ be the following event
\[
H_n = \{|\dot{Y}(\tau_{n+1}) - \dot{Y}(\tau_{n}) -
\int_{\tau_n}^{\eta_{n+1}} \widetilde{{F}}_{n}(z_n(s)) d s | \leq
|v_n|^{-1 + \delta}\}.
\]
The following Lemma will be proved in the Appendix.
\begin{lemma} \label{nlnl}
For each $\delta > 0$ we have
\begin{equation}
\label{star} \mathrm{E}\left( \chi_{H_0} |\dot{Y}(\tau_{1}) - v_0
- \int_0^{\eta_1} {F}_{0}(z_0(s)) d s  | \right)\leq |v_0|^{-3 +
\delta}
\end{equation}
for all sufficiently large $|v_0|$.

For each $\delta > 0$ there is $r > 0$ such that for each $n$
\begin{equation} \label{bb5}
{\mathrm{E}}\left( \chi_{H_n}  |\dot{Y}(\tau_{n+1}) -
\dot{Y}(\tau_{n}) - \int_{\tau_n}^{\eta_{n+1}}
\widetilde{{F}}_{n}(z_n(s) d s) |  | \mathcal{G}_{n}\right) \leq
  |v_n|^{-3 + \delta}
\end{equation}
almost surely on the event $|v_n| > r$.

\end{lemma}
 \subsection{Behavior of $Y(t)$ and $\dot{Y}(t)$ on a Time
Interval Proportional to $|v_0|^3$}

 Recall that $x(t) = X(t/|v_0|)$ satisfies (\ref{small})
with $\varepsilon = 1/|v_0|^2$ and initial data $x(0) = 0,
\dot{x}(0) = v_0/|v_0|$. As discussed in the Introduction, the
scale on which we see the diffusion for $\dot{x}(t)$ is of order
$t \sim 1/\varepsilon^2 = |v_0|^4$. Since  $\dot{x}(|v_0| t) =
\dot{X}(t)/|v_0|$, one needs time of order $t \sim |v_0|^3$ to see
fluctuations of order one for the process $\dot{X}(t)/|v_0|$.

In this section we recall the effective equation for $\dot{Y}$ on
the scale $|v_0|^3$ and provide estimates for the probability that
$\dot{Y}$ changes much faster or much slower than expected.

%For an arbitrary
%$N$ and large enough $|v_0|$, we shall bound the probabilities of
%several `adverse events' by $|v_0|^{-N}$. An example of an
%`adverse event' is that the velocity experiences a large
%fluctuation (of order $|v_n|^{-1+\delta}$ with $\delta > 0$) on
%one of the
% time intervals $[\tau_n, \tau_{n+1}]$, where $\tau_n$ is smaller
%than the time which is needed for the velocity to change by $c
%|v_0|$.

For $a, r > 0$, $b > 1$, and $n \geq 0$, let
\[
\hat{\tau}^a_n =  \min_{k \geq n}\{\tau_k: \tau_k - \tau_n \geq
ar^3 \},
\]
\[
\check{\tau}^b_n = \min_{k \geq n}\{\tau_k: | \dot{Y}(t)| \notin
(r,b r) ~~{\rm for}~{\rm some}~\tau_n \leq t \leq \tau_k \},
\]
\[
\overline{\tau}_n = \min\{ \hat{\tau}^a_n, \check{\tau}^b_n \}.
\]
%\[
%\overline{\tau}_n = \min\{\tau_k: \tau_k - \tau_n > ar^3~~{\rm
%or}~~| \dot{Y}(t)| \notin (r , b r) ~~{\rm for}~{\rm some}~\tau_n
%\leq t \leq \tau_k \},
%\]
%\[
%\overline{\overline{\tau}}_n = \min_{k \geq n}\{\tau_k: |
%\dot{Y}(t)| \notin (r,br) ~~{\rm for}~{\rm some}~\tau_n \leq t
%\leq \tau_k \}.
%\]
%\[
%\overline{\tau} = \overline{\tau}(a, b, r) = \min\{\tau_n: \tau_n
%> a r^3~~{\rm or}~~| \dot{Y}(t)| \notin (r, b r) ~~{\rm for}~{\rm
%some}~t \leq \tau_n \}.
%\]
(As always, the minimum over the empty set is $+ \infty$.) In what
follows $a$ and $b$ will be fixed.  The constant $r$ will serve as
a large parameter, and $|v_0|$ will be assumed to be of order $r$.
Thus $\hat{\tau}^a_0$ is the first of the stopping times $\tau_k$
which is larger than $a r^3$. Roughly speaking, $\hat{\tau}^a_0$
is very close to $a r^3$. The stopping time $\check{\tau}^b_0$ is,
roughly speaking, the first time when $| \dot{Y}(t)|$ changes from
$|v_0|$ to either $r$ or $b r$ (assuming that $|v_0| \in (r, b
r)$).

%At first we shall study the behavior of $Y(t)$ for large $r$,
%$|v_0| \in (r,a r)$, and $ t \leq \overline{\tau}$. In particular,
%we shall demonstrate that with high probability $\dot{Y}(t)$  does
%not change significantly, compared to $\dot{Y}(\tau_n)$,  between
%$\tau_n$ and $\tau_{n+1}$ if $\tau_n < \overline{\tau}$.
%
%Throughout this section we shall assume that $|v_0|$ is large and
%$ t \leq  c |v_0|^3$, where $c$ is a constant independent of
%$v_0$. Let $0 < r_1 < 1 < r_2 < \infty$. We shall stop observing
%the process $Y(t)$ earlier  than at time $c |v_0|^3$  if $ |
%\dot{Y}(t)|$ becomes equal to either  $r_1 |v_0|$ or $r_2 |v_0|$
%at  $t < c |v_0|^3$.

Assume that $r$ is large and $|v_0| \in (r,b r)$. Let us first
describe the behavior of the process $Y(t)$ on the time interval
$[0, \overline{\tau}_{0}]$.
\begin{lemma} \label{fl}
For each $N$, $\delta > 0$, $a > 0$, and $b >1$,  we have
\begin{equation} \label{cc1}
{\mathrm{P}}\left( \tau_{n+1} - \tau_n > |v_n|^{-1+\delta}~{\it
for}~{\it some}~n~{\it such}~{\it that}~\tau_n < \overline{\tau}_0
\right) \leq r^{-N},
\end{equation}
\begin{equation} \label{cc2}
{\mathrm{P}}\left( \sup_{\tau_n \leq t \leq \tau_{n+1}}
|\dot{Y}(t) - v_n|
> |v_n|^{-1 + \delta}~{\it for}~{\it some}~n~{\it such}~{\it that}~\tau_n <
\overline{\tau}_0 \right) \leq r^{-N},
\end{equation}
\begin{equation} \label{cc5}
{\mathrm{P}}\left(\overline{\tau}_0 = \infty \right) \leq r^{-N}
\end{equation}
for all sufficiently large $r$ (i.e. for all $r \geq r_0$, where
$r_0$ depends on the distribution of the force field and on $N$,
$\delta$, $a$ and $b$) and all $|v_0| \in (r, b r)$.
\end{lemma}
\proof For  fixed $n$, the probability ${\mathrm{P}}\left(
\tau_{n+1} - \tau_n > |v_n|^{-1+\delta},~\tau_n <
\overline{\tau}_0 \right)$ is estimated from above by $r^{-N}$ due
to  (\ref{aa1}) (if $n = 0$) and (\ref{bb1}) (if $n \geq 1$). The
number of $n$ for which $ \tau_n < \overline{\tau}_0$ does not
exceed $ abr^{4}$.  Since $N$ was arbitrary, this implies
(\ref{cc1}). In the same way, (\ref{aa2}) and (\ref{bb2}) imply
(\ref{cc2}). Finally, (\ref{cc1}) implies (\ref{cc5}) again due to
the fact that $\tau_n \geq \overline{\tau}_0$ for $n > abr^{4}$.
\qed

As before, by considering $Z(t)$ instead of $Y(t)$, we obtain the
following.
\begin{corollary}
For each $N$, $\delta > 0$, $a > 0$, and $b >1$,  we have
\[
{\mathrm{P}}\left( \tau_{k+1} - \tau_k > |v_k|^{-1+\delta}~{\it
for}~{\it some}~k~{\it such}~{\it that}~\tau_n < \tau_k <
\overline{\tau}_n | \mathcal{G}_n \right) \leq r^{-N},
\]
\[
{\mathrm{P}}\left( \sup_{\tau_k \leq t \leq \tau_{k+1}}
|\dot{Y}(t) - v_k|
> |v_k|^{-1 + \delta}~{\it for}~{\it some}~k~{\it such}~{\it that}~\tau_n \leq \tau_k <
\overline{\tau}_n  | \mathcal{G}_n \right) \leq r^{-N},
\]
\[
{\mathrm{P}}\left(\overline{\tau}_n = \infty | \mathcal{G}_n
\right) \leq r^{-N}
\]
for all sufficiently large $r$ almost surely on the event $|v_n|
\in (r, b r) $.
\end{corollary}

\begin{lemma} \label{ole}
Assume
that $v_0 = (|v_0|,0,...,0)$, and $|v_0| \rightarrow \infty$. Then
both families of processes ${  \dot{Y}( |v_0|^3 t)}/{|v_0|}$ and
${ \dot{Z}( |v_0|^3 t)}/{|v_0|}$ converge weakly to the diffusion
process $\overline{V}(t)$ given by (\ref{diff1}) starting at
$(1,0,...,0)$.
\end{lemma}
This lemma is a slight modification of the results of \cite{DGL,
KP, KR} to the case of the processes $Y(t)$ and $Z(t)$, so we omit
the proof. For example, the main theorem of \cite{KP} on page 24
gives the desired result, except the fact that in the setting of
\cite{KP} there is no renewal of the force field. The proof,
however, goes through without major modifications.

\begin{corollary} \label{fs} For each  $a > 0$ and $b >1$ there is $c < 1$
such that
\[
{\mathrm{P}} \left(   \hat{\tau}^a_0 < \check{\tau}^b_0 \right)
\leq c
\]
for all sufficiently large $r$ and all $|v_0| \in (r, b r)$. The
same is true if $\hat{\tau}^a_0$ and  $\check{\tau}^b_0$ are
defined as the stopping times for the process $Z(t)$ with initial
velocity $v_0$.
\end{corollary}
 \proof It is sufficient to consider the process
$Y(t)$ since the proof for the process $Z(t)$ is completely
similar. From Lemma~\ref{ole} and the rotation-invariance of the
force field it follows that when $|v_0| \rightarrow \infty$, the
processes ${| \dot{Y}( |v_0|^3 t)|}/{|v_0|}$ converge weakly to
the diffusion process $|\overline{V}(t)|$ with $|\overline{V}(0)|
= 1$. Therefore
\[
\limsup_{|v_0| \rightarrow \infty} {\mathrm{P}}
\left(\frac{|v_0|}{b} < {| \dot{Y}( t)|} < |v_0| b~~{\rm for}~{\rm
all}~0 \leq t \leq a |v_0|^3 \right) =
\]
\[
\limsup_{|v_0| \rightarrow \infty} {\mathrm{P}} \left(\frac{1}{b}
< {| \dot{Y}( |v_0|^3 t)|}/{|v_0|} < b~~{\rm for}~{\rm all}~0 \leq
t \leq a \right) \leq
\]
\[
 {\mathrm{P}} \left(\frac{1}{2b} < |\overline{V}(t)| < 2b~~{\rm
for}~{\rm all}~0 \leq t \leq a \right) < 1,
\]
where the first inequality holds  due to the weak convergence of
$\dot{Y}( |v_0|^3 t)|/{|v_0|}$ to $\overline{V}(t)$ since the
closure of the set  $\{\varphi \in C([0,a],\mathbb{R}): 1/b <
\varphi(t) < b~~{\rm for}~{\rm all}~0 \leq t \leq a\}$ is
contained in the open set  $\{\varphi \in C([0,a],\mathbb{R}):
1/2b < \varphi(t) < 2b~~{\rm for}~{\rm all}~0 \leq t \leq a\}$.
The second inequality is due to the fact that $|\overline{V}(t)|$
is a non-degenerate diffusion process on $(0, \infty)$ starting at
$1$, as follows from (\ref{diff1}) and Lemma~\ref{ole}. The
Corollary now follows from the definitions of $\hat{\tau}^a_0$ and
$\check{\tau}^b_0$. \qed

 Now we can  replace the stopping time $ \overline{\tau}_0$
by $\check{\tau}^b_0$ in Lemma~\ref{fl}.
 \begin{lemma} \label{wweel}
For each $N$, $\delta > 0$,  and $b >1$,  we have
\begin{equation} \label{dd1}
{\mathrm{P}}\left( \tau_{n+1} - \tau_n > |v_n|^{-1+\delta}~{\it
for}~{\it some}~n~{\it such}~{\it that}~\tau_n <\check{\tau}^b_0
\right) \leq r^{-N},
\end{equation}
\begin{equation} \label{dd2}
{\mathrm{P}}\left( \sup_{\tau_n \leq t \leq \tau_{n+1}}
|\dot{Y}(t) - v_n|
> |v_n|^{-1 + \delta}~{\it for}~{\it some}~n~{\it such}~{\it that}~\tau_n <
\check{\tau}^b_0 \right) \leq r^{-N},
\end{equation}
\begin{equation} \label{dd5}
{\mathrm{P}}\left(\check{\tau}^b_0 = \infty \right) \leq r^{-N}
\end{equation}
for all sufficiently large $r$ and all $|v_0| \in (r, b r)$.
\end{lemma}
\proof Let
\[
q_Y = q_Y(r)=\sup_{v_0: |v_0| \in (r, b r) } {\mathrm{P}}\left(
\tau_{n+1} - \tau_n > |v_n|^{-1+\delta}~{\it for}~{\it
some}~n~{\it such}~{\it that}~\tau_n <\check{\tau}^b_0 \right).
\]
Let $q_Z=q_Z(r)$ be defined as $q_Y$, with the only difference that the
stopping times are assumed to correspond to the process $Z(t)$
instead of $Y(t)$. Take an arbitrary $a > 0$. Then, for $ |v_0|
\in (r, b r)$ we have
\[
{\mathrm{P}}\left( \tau_{n+1} - \tau_n > |v_n|^{-1+\delta}~{\it
for}~{\it some}~n~{\it such}~{\it that}~\tau_n <\check{\tau}^b_0
\right) \leq
\]
\[
{\mathrm{P}}\left( \tau_{n+1} - \tau_n > |v_n|^{-1+\delta}~{\it
for}~{\it some}~n~{\it such}~{\it that}~\tau_n <\overline{\tau}_0
\right) +
\]
\[
{\mathrm{P}}\left( \tau_{n+1} - \tau_n > |v_n|^{-1+\delta}~{\it
for}~{\it some}~n~{\it such}~{\it that}~\overline{\tau}_0 \leq
\tau_n <\check{\tau}^b_0 \right).
\]
The first term in the right-hand side does not exceed $r^{-N}$ by
Lemma~\ref{fl}. In order to estimate the second term, we observe
that
\[
{\mathrm{P}}\left( \tau_{n+1} - \tau_n > |v_n|^{-1+\delta}~{\it
for}~{\it some}~n~{\it such}~{\it that}~\overline{\tau}_0 \leq
\tau_n <\check{\tau}^b_0 \right)=
\]
\[
{\mathrm{P}}\left( \hat{\tau}^a_0  <\check{\tau}^b_0~~{\rm and}~~
\tau_{n+1} - \tau_n > |v_n|^{-1+\delta}~{\it for}~{\it
some}~n~{\it such}~{\it that}~\hat{\tau}^a_0 \leq \tau_n
<\check{\tau}^b_0 \right) \leq
\]
\[
 {\mathrm{P}} \left( \hat{\tau}^a_0 < \check{\tau}^b_0 \right)
 q_Z,
\]
where the inequality is due to the Markov property with respect to
the stopping time $\hat{\tau}^a_0$ (see formula~(\ref{smp2})).
Therefore, by Corollary~\ref{fs},
\[
q_Y \leq r^{-N} + c q_Z.
\]
Similarly,
\[
q_Z \leq r^{-N} + c q_Z.
\]
Since $c <1$ and $N$ is arbitrary, these two inequalities imply
(\ref{dd1}). The proof of (\ref{dd2})  is similar. In order to
prove  (\ref{dd5}), define
\[
\overline{q}_Y(k) = \sup_{v_0: |v_0| \in (r, b r) }
{\mathrm{P}}\left( \check{\tau}^b_0 > 2kar^3 \right),~~k \geq 0.
\]
Let $\overline{q}_Z(k)$ be defined as $\overline{q}_Y(k)$, with
the only difference that the stopping times are assumed to
correspond to the process $Z(t)$ instead of $Y(t)$. Note that for
$ |v_0| \in (r, b r)$ we have
\begin{equation} \label{cor1}
{\mathrm{P}}\left( \check{\tau}^b_0 > 2kar^3 \right) \leq
{\mathrm{P}}\left( \hat{\tau}^a_0 > 2ar^3, \check{\tau}^b_0
> 2kar^3  \right) +
{\mathrm{P}}\left( \hat{\tau}^a_0 \leq 2ar^3, \check{\tau}^b_0
> 2kar^3 \right),~k \geq 1.
\end{equation}
 Note that if $ \hat{\tau}^a_0, \check{\tau}^b_0 > 2ar^3 $, then
$ \tau_{k+1} - \tau_k > ar^3$ for some $k$ with $\tau_k <
\overline{\tau}_0$. Therefore, the first term on the right-hand
side can be estimated from above by $r^{-N}$ due to (\ref{cc1}).
By the Markov property with respect to the stopping time
$\hat{\tau}^a_0$ (see formula~(\ref{smp2})),  the second term does
not exceed $ {\mathrm{P}} \left( \hat{\tau}^a_0 < \check{\tau}^b_0
\right) \overline{q}_Z(k-1) \leq  c \overline{q}_Z(k-1)$, where
$c$ is the constant from Corollary~\ref{fs}. Therefore,
\[
\overline{q}_Y(k) \leq r^{-N} + c \overline{q}_Z(k-1).
\]
Similarly,
\[
\overline{q}_Z(k) \leq r^{-N} + c \overline{q}_Z(k-1).
\]
Since $c <1$ and $N$ is arbitrary, these two inequalities imply
that
\begin{equation} \label{iiu}
\max(\overline{q}_Y(k), \overline{q}_Z(k)) \leq r^{-N} + c^k.
\end{equation}
This implies (\ref{dd5}) since an arbitrarily large $k$ can be
taken.
 \qed
\begin{corollary} \label{fie}
For each $N$, $\delta > 0$, and $b >1$,  we have
\[
{\mathrm{P}}\left( \tau_{k+1} - \tau_k > |v_k|^{-1+\delta}~{\it
for}~{\it some}~k~{\it such}~{\it that}~\tau_n < \tau_k <
\check{\tau}^b_n | \mathcal{G}_n \right) \leq r^{-N},
\]
\[
{\mathrm{P}}\left( \sup_{\tau_k \leq t \leq \tau_{k+1}}
|\dot{Y}(t) - v_k|
> |v_k|^{-1 + \delta}~{\it for}~{\it some}~k~{\it such}~{\it that}~\tau_n \leq \tau_k <
\check{\tau}^b_n  | \mathcal{G}_n \right) \leq r^{-N},
\]
\[
{\mathrm{P}}\left(\check{\tau}^b_n = \infty | \mathcal{G}_n
\right) \leq r^{-N}
\]
for all sufficiently large $r$ almost surely on each of the events
$|v_n| \in (r, b r) $.
\end{corollary}

\begin{lemma} \label{ccf}
 For each $N$, $ \delta > 0$, and $k >
0$ we have
\[
{\mathrm{P}}\left( \sup_{0 \leq t\leq |v_0|^{3 - \delta}} |
\dot{Y}(t) - v_0| > k |v_0|\right) \leq |v_0|^{-N}
\]
for all sufficiently large $|v_0|$.
\end{lemma}
\proof Let us write
\begin{equation} \label{yy}
\dot{Y}(t) - v_0 = (v_1 - v_0) + (v_2 - v_1)+...+ (v_n - v_{n-1})
+ \dot{Y}(t) - v_n,
\end{equation}
where $n = n(t)$ is the random time such that $\tau_{n-1} \leq t <
\tau_n$. Without loss of generality we may assume that $k \leq
1/2$. Let $L = L(v_0) = [2(|v_0|^{4-\delta} l^{-1} +1)]$, where $l
= 4R+1$ is the constant used in the definition of the stopping
times $\tau_n$.
 Let $\sigma$ be the random time defined by
\[
\sigma = \min\{m: |(v_1 - v_0) + (v_2 - v_1)+...+ (v_m - v_{m-1})|
\geq  k |v_0|/2\} \wedge L.
\]
%Let $S_m$ be the event $\sigma = m$.
%\[
%S_m = \{|(v_1 - v_0) + (v_2 - v_1)+...+ (v_m - v_{m-1})| \geq  k
%|v_0|/2\} \cap
%\]
%\[
%\cap \{|(v_1 - v_0) + (v_2 - v_1)+...+ (v_i - v_{i-1})| <  k
%|v_0|/2~~{\rm for}~i < m\}.
%\]
%Let
%\[
%\overline{S} = \bigcup_{m=1}^{L} S_m.
%\]
Observe that if $| \dot{Y}(t)|$ does not exceed $2 |v_0|$ on the
time interval $[0,|v_0|^{3 - \delta}]$, then $\tau_{ L} \geq
|v_0|^{3 - \delta}$ since $\tau_{n+1} - \tau_n \geq l/(2 |v_0|)$
for each $n $ such that $\tau_n < |v_0|^{3 - \delta}$, as follows
from the definition of the stopping times $\tau_n$. Therefore,
\[
\{  \sup_{0 \leq t\leq |v_0|^{3 - \delta}} | \dot{Y}(t) - v_0| > k
|v_0| \} \subseteq
\]
\[
\subseteq \{\sigma < L \}
 \cup\left( \{ \sigma = L \} \cap \{  \sup_{0 \leq t\leq |v_0|^{3 - \delta}} | \dot{Y}(t) - v_0| > k
|v_0| \} \right)
\]
\[
\subseteq \{\sigma < L \}
 \cup\left( \{ \sigma = L \} \cap \bigcup_{m=1}^{L}\{ \sup_{\tau_{m-1} \leq t\leq
\tau_m} | \dot{Y}(t) - v_0| \geq k |v_0| \} \right)
\]
\[
\subseteq  \{\sigma < L \}
 \cup\left( \{ \sigma = L \} \cap \bigcup_{m=1}^{L}\{ \sup_{\tau_{m-1} \leq t\leq
\tau_m} | \dot{Y}(t) - v_{m-1}| \geq k |v_0|/2 \} \right).
\]
% Let us estimate the probability of
%$S_m$ for $m \leq L$.
Define $c_i$, $i \geq 1$, by
\[
(v_{i} - v_{i-1}) =
\int_{\tau_{i-1}}^{\eta_i}\widetilde{F}_{i-1}(z_{i-1}(s)) d s +
c_{i}.
 \]
By Corollary~\ref{fttt} and Lemma~\ref{nlnl} (formulas
(\ref{bb3}), (\ref{bb4}), and (\ref{bb5})), for each  $N$ and
$\varepsilon > 0$ the estimates
\begin{equation} \label{fc1}
\mathrm{P}(|c_{i}| > |v_0|^{-1+\varepsilon} | \mathcal{G}_{{i-1}}
) \leq |v_0|^{-N},
\end{equation}
\begin{equation} \label{fc2}
\mathrm{E}( |c_i|  \chi_{ \{ |c_i| \leq |v_0|^{-1  + \varepsilon}
\}} | \mathcal{G}_{{i-1}} ) \leq |v_0|^{-3+\varepsilon}
\end{equation}
hold for each $i$ on $\{ \sigma \geq i \}$  if $|v_0|$ is
sufficiently large. Let
\[
C_j = \sum_{i =1}^{j \wedge \sigma} |c_i|,
\]
\begin{equation} \label{formhh}
h_j =  \sum_{i =1}^{j \wedge \sigma} (|c_i| \chi_{ \{ |c_i| \leq
|v_0|^{-1 + \varepsilon} \}} -  |v_0|^{-3+  \varepsilon} ).
\end{equation}
By (\ref{fc1}) and (\ref{fc2}), $h_j$ is a supermartingale. Let
$h_j = \alpha_j + \beta_j$ be the Doob decomposition of $h_j$,
where
\[
\beta_j = \sum_{i =1}^{j \wedge \sigma} \mathrm{E}\left( (|c_i|
\chi_{ \{ |c_i| \leq |v_0|^{-1  + \varepsilon} \}} -  |v_0|^{-3+
\varepsilon} )| \mathcal{G}_{{i-1}} \right)
\]
is a non-increasing process.  Let $\langle \alpha \rangle_j$ be
the quadratic variation of $\alpha_j$. It is equal to
\[
\langle \alpha \rangle_j = \sum_{i=1}^{j} \mathrm{E} \left( ((h_i
- h_{i-1}) - (\beta_i - \beta_{i-1}) )^2 | \mathcal{G}_{{i-1}}
\right),~~ j \geq 1.
\]
From (\ref{formhh}) it follows that $|h_i - h_{i-1}| \leq
2|v_0|^{-1+  \varepsilon}$ for sufficiently large $|v_0|$, and
consequently $|\beta_i - \beta_{i-1}| \leq 2|v_0|^{-1+
\varepsilon}$. Therefore $\langle \alpha \rangle_j \leq 16 j
|v_0|^{-2+  2\varepsilon}$, which implies that for each $p \in
\mathbb{N}$ there is a constant $k_p$ such that
\[
 \langle \alpha \rangle_j^p \leq  k_p (j |v_0|^{-2 + 2
\varepsilon })^p.
\]
Applying this inequality to $j = L$ and noting that $\sigma \leq
L$ and $\langle \alpha \rangle_j^p$ is non-decreasing in $j$, we
 obtain
 \[
 \langle \alpha \rangle_\sigma^p \leq  k'_p ( |v_0|^{2 +
2 \varepsilon  - \delta})^p.
 \]
Take $\varepsilon = \delta/3$. Then, by the Chebyshev Inequality
and the Martingale Moment Inequality, for each $N$ there are $p$
and $K_p$ such that
\[
\mathrm{P}(h_\sigma \geq k|v_0|/8 ) \leq \mathrm{P}(\alpha_\sigma
\geq k|v_0|/8 ) \leq \mathrm{P}(|\alpha_\sigma|^{2p} \geq
(k|v_0|/8)^{2p} ) \leq
\]
\[
\frac{ \mathrm{E} |\alpha_\sigma|^{2p}}{(k|v_0|/8)^{2p}} \leq
\frac{ K_p \mathrm{E} \langle \alpha
\rangle_\sigma^{p}}{(k|v_0|/8)^{2p}} \leq |v_0|^{-N}
\]
if $|v_0|$ is sufficiently large. Note that
\[
\mathrm{P}(C_\sigma \geq k|v_0|/4 ) \leq
\]
\begin{equation} \label{mdl}
\mathrm{P}(h_\sigma \geq k|v_0|/8 ) + \mathrm{P}(C_\sigma -
h_\sigma \geq k|v_0|/8 ) \leq
\end{equation}
\[
|v_0|^{-N} + \mathrm{P}( \sum_{i =1}^{ \sigma} (|c_i| \chi_{ \{
|c_i| > |v_0|^{-1 + \varepsilon} \}} +  |v_0|^{-3+ \varepsilon} )
\geq k|v_0|/8 ).
\]
Since $\sigma \leq L$, and therefore ${ \sigma} |v_0|^{-3+
\varepsilon}  < k|v_0|/8$ for all sufficiently large $|v_0|$, the
second term in the right hand side of (\ref{mdl}) is estimated
from above by
\[
\mathrm{P}(|c_i| > |v_0|^{-1 + \varepsilon}~~{\rm for}~{\rm
some}~1 \leq i \leq \sigma) \leq L |v_0|^{-N},
\]
where the inequality follows after integrating both sides of
(\ref{fc1}) over the event $\{ \sigma \geq i \}$ and recalling
that $\sigma \leq L$. Since $N$ was arbitrary and $L \leq |v_0|^4$
for all sufficiently large $|v_0|$, this implies that
\begin{equation} \label{ccap}
\mathrm{P}(C_\sigma \geq k|v_0|/4 ) \leq  |v_0|^{-N}
\end{equation}
if $|v_0|$ is sufficiently large.
% Therefore,
%\begin{equation} \label{cccc}
%\mathrm{P}(S_m \cap \{|c_0+...+c_{m-1}| > 2 |v_0|^{1-\delta/2} /k
%\})\leq |v_0|^{-N}
%\end{equation}
%holds for any $N$ if $|v_0|$ is sufficiently large.
Let
\begin{equation}
\label{SVMT} f_j = \sum_{i=1}^{j \wedge \sigma}
\int_{\tau_{i-1}}^{\eta_i}\widetilde{F}_{i-1}(z_{i-1}(s)) d s,~~ j
\geq 1.
\end{equation}
Notice that $f_j$ is a martingale vector. Indeed,
\[
\mathrm{E}( f_{j} - f_{j-1}|\mathcal{G}_{{i-1}}) = \chi_{\{j \leq
\sigma\}} \mathrm{E}(
\int_{\tau_{j-1}}^{\eta_j}\widetilde{F}_{j-1}(z_{j-1}(s)) d s
|\mathcal{G}_{{i-1}}).
\]
Let $\widetilde{\mathcal{G}}_n$ be the $\sigma$-algebra determined
by the Poisson field $r^n_i$ (see the definition of the random
field $F_n$ in Section~\ref{tdis}). Then the last conditional
expectation can be written as
\[
\mathrm{E}(
\int_{\tau_{j-1}}^{\eta_j}\widetilde{F}_{j-1}(z_{j-1}(s)) d s
|\mathcal{G}_{{i-1}}) =  \mathrm{E}( \mathrm{E}(
\int_{\tau_{j-1}}^{\eta_j}\widetilde{F}_{j-1}(z_{j-1}(s)) d s|
\sigma(\mathcal{G}_{{i-1}} \cup \widetilde{\mathcal{G}}_{{i-1}}
))|\mathcal{G}_{{i-1}} )
\]
The inner conditional expectation is equal to zero since the
functions $f^n_i$ from the definition of the fields
$\widetilde{F}_n$ are independent of the Poisson fields, and are
symmetrically distributed by (\ref{symmetr}).

We shall denote the components of the vector $f_j$ by $f_j^a$, $1
\leq a \leq d$.  The quadratic variation of $f_j$ is
\[
\langle f^a, f^b \rangle_j = \sum_{i=1}^{j \wedge \sigma}
\mathrm{E} \left(
(\int_{\tau_{i-1}}^{\eta_i}\widetilde{F}^a_{i-1}(z_{i-1}(s)) d
s)(\int_{\tau_{i-1}}^{\eta_i}\widetilde{F}^b_{i-1}(z_{i-1}(s)) d
s) | \mathcal{G}_{{i-1}} \right),~~ j \geq 1.
\]
 Using arguments similar to those
in the proof of Lemma~\ref{ftt01} (see the justification of
formula (\ref{pp2a11})), it is easy to show that for
 $p \in \mathbb{N}$  there is a constant $k''_p$ such that for all $j \leq
L$ we have
\[
\mathrm{E} |\langle f \rangle_j|^p \leq k''_p(j |v_0|^{-2 +
\delta/2})^p
\]
%
%\[
%\mathrm{P}(S_m \cap \{ \langle f \rangle_j > j |v_0|^{-2 + \delta/2} \} ) \leq
%|v_0|^{-N}
%\]
 if $|v_0|$ is sufficiently large, where $|\langle f \rangle_j|$ stands for
 the norm of the matrix $\langle f^a, f^b \rangle_j$. In particular, for $j = L$ we
 obtain
 \[
\mathrm{E} |\langle f \rangle_\sigma|^p \leq k'''_p(|v_0|^{2 -
\delta/2})^p.
 \]
%
%\[
%\mathrm{P}(S_m \cap \{ \langle f \rangle_m >  3|v_0|^{2 - \delta/2} \} ) \leq
%|v_0|^{-N}.
%\]
%Therefore, as follows from the Martingale Moment Inequality,
%\[
%\mathrm{P}(S_m \cap \{|f_m| >  k|v_0|/4\})  \leq |v_0|^{-N}.
%\]
By the Chebyshev Inequality and the Martingale Moment Inequality,
for each $N$ there are $p \in \mathbb{N}$ and $K_p >0$ such that
\[
\mathrm{P}(|f_\sigma| \geq k|v_0|/4 ) = \mathrm{P}(|f_\sigma|^{2p}
\geq (k|v_0|/4)^{2p} ) \leq \frac{ \mathrm{E}
|f_\sigma|^{2p}}{(k|v_0|/4)^{2p}} \leq \frac{ K_p \mathrm{E}
|\langle f \rangle_\sigma|^{p}}{(k|v_0|/4)^{2p}} \leq |v_0|^{-N}
\]
if $|v_0|$ is sufficiently large.

 Together
with (\ref{ccap}), this implies that
%$ \mathrm{P}(S_m) \leq
%|v_0|^{-N}$. Since $N$ was arbitrary, we can take a sum over $1
%\leq m \leq 2|v_0|^{4-\delta}+1$ and obtain that
\[
\mathrm{P}(\sigma < L )  \leq |v_0|^{-N}.
\]
It easily follows from Corollary~\ref{fttt} that
\[
\mathrm{P}( \{ \sigma = L \} \cap \bigcup_{m=1}^{L}\{
\sup_{\tau_{m-1} \leq t\leq \tau_m} | \dot{Y}(t) - v_{m-1}| \geq k
|v_0|/2 \} )  \leq |v_0|^{-N}.
\]
 \qed

\section{Long time behavior of $Y(t).$} \label{ltbe}
The goal of this section is to show that the paths of $Y(t)$ are
not self-intersecting with probability close to one, and therefore
the distributions of $X$ and $Y$ are close, as claimed. This is
achieved in subsection \ref{SSSI}. In subsection \ref{SSDerGr} we
establish some {\it a priori} bounds on the growth of $|\dot{Y}|$.

\subsection{Behavior of $|\dot{Y}(t)|$ as $t \rightarrow \infty$.}
\label{SSDerGr}
In this section we shall demonstrate that for large $|v_0|$ with
high probability the norm of the velocity vector $|\dot{Y}(t)|$
grows as $t^{1/3}$ when $t \rightarrow \infty$.

The idea of the proof is the following. We consider $\dot{Y}$ at
the moments $s_n$ its modulus crosses $2^l$ (alternating odd and
even $l$). By Lemma \ref{ole} and \eqref{BiasUp},
$\ln|\dot{Y}(s_n)|$ can be well approximated by a simple random
walk biased to the right. It follows that $|\dot{Y}(s_n)|$ grows
exponentially and so $|\dot{Y}(t)|$ spends most of the time near
its maximum. By Lemma \ref{ole}, $s_{n+1}-s_n$ is of order
$|\dot{Y}(s_n)|^3$, which implies the desired result. Let us now
give a detailed proof.

We start by describing a discretized  version of the process $|
\dot{Y}(t)|$. Let $2^{m - \frac{1}{2}} \leq |v_0| <
2^{m+\frac{1}{2}}$ for some $m \in \mathbb{Z}$. Let $0 < \delta <
1 $. Define, inductively, a sequence of events $
\mathcal{E}_n^\delta$ and three processes $s_n, t_n \in
\mathbb{R}^+ \cup \infty$ and $\xi_n \in \mathbb{Z}$ as follows.
Let $\mathcal{E}_0^\delta = {\Omega}$, $s_0 = t_0 = 0$, and $\xi_0
= m$.
% Let $s_1$ be the first time when $ |
%\dot{Y}(t)|$ reaches either $2^{m-1}$ or $2^{m+1}$, that is
%\[
%s_1 = \inf\{t:| \dot{Y}(t)| = 2^{m-1}~{\rm or}~| \dot{Y}(t)| =
%2^{m+1} \}.
%\]
% Let
%$t_1$ be the first of the stopping times $\tau_k$ which follows
%$s_1$, that is
%\[
%t_1  = \min\{ \tau_k: \tau_k \geq s_1 \}.
%\]
%Let $\xi_1 = \log_2| \dot{Y}(s_1)|$ (we can assign  $\xi_1$ an
%arbitrary value if $s_1 =\infty$). Let
%\[
%\mathcal{E}_1 = \{t_1 < \infty \} \cap \{ \tau_{k+1} - \tau_k \leq
%|v_k|^{-1+\delta} {\rm for}~{\rm all}~k~{\rm such}~{\rm that}~0 <
%\tau_k \leq t_1  \}
%\]
%
%
%\[
%\mathcal{E}_1 = \{t_1 < \infty \} \cap \{  \tau_{k} \leq
%\min(\tau'_{k},
% \tau_{k}'' )~{\rm for}~{\rm all}~k~{\rm such}~{\rm that}~0 < \tau_k \leq t_1  \}.
%\]
 Assume that $\mathcal{E}_{n-1}^\delta$, $s_{n-1}$,
$t_{n-1}$, and $\xi_{n-1}$ have been defined for some $n \geq 1$.
We then define
\[
s_n = \inf\{t:| \dot{Y}(t)| = 2^{\xi_{n-1}-1} ~{\rm or}~|
\dot{Y}(t)| = 2^{\xi_{n-1}+1} \},
\]
\[
t_n  = \min\{ \tau_k: \tau_k \geq s_n \},~~{\rm and}~~\xi_n =
\log_2| \dot{Y}(s_n)|.
\]
%\[
%\mathcal{E}_n = \mathcal{E}_{n-1} \cap \{ t_n < \infty \} \cap \{
%\tau_{k} \leq \min(\tau'_{k},
% \tau_{k}'' )~{\rm for}~{\rm all}~k~{\rm such}~{\rm that}~t_{n-1} < \tau_k \leq t_n  \}.
%\]
\[
\mathcal{E}_n^\delta = \mathcal{E}_{n-1}^\delta \cap \{ t_n <
\infty \} \cap
 \{ \tau_{k+1} - \tau_k \leq
|v_k|^{-1+\delta}~{\rm for}~{\rm all}~k~{\rm such}~{\rm that}~
t_{n-1} < \tau_k \leq t_n  \} \cap
\]
\[
\cap \{ \sup_{\tau_k \leq t \leq \tau_{k+1}} |\dot{Y}(t) - v_k|
\leq |v_k|^{-1 + \delta}~{\rm for}~{\rm all}~k~{\rm such}~{\rm
that}~t_{n-1} \leq \tau_k \leq t_n \}.
\]

Let $ \mathcal{F}_n$ be the $\sigma$-algebra of events determined
before $t_n$, that is $ \mathcal{F}_n = \sigma (\cup_{m:  \tau_m
\leq t_n} \mathcal{G}_m)$. The process $\xi_n$ can be viewed as a
random walk (with memory and random transition times), while $t_n$
can be viewed as transition times for the random walk. Note that
the process $|\dot{Y}(t)|$ takes values equal to powers of 2 at
times $s_n$. It is more convenient, however, to consider times
$t_n$ (which are close to times $s_n$, but coincide with the
stopping times $\tau_k$), and the $\sigma$-algebras  $
\mathcal{F}_n$ are defined using the times $t_n$.

 The following  lemma describes the one-step transition times
and transition probabilities.
\begin{lemma} \label{qqaax}

%\item For any $r >0$ there exist $M$ and $c > 0$ such that for $m \geq M$ we
%have
%\[
%\mathrm{P}(\mathcal{E}_n \cap \mathcal{E}^r_n | \mathcal{F}_{n-1})
%\geq 1  - e^{-c m} ~~{\it almost}~{\it surely}~{\it on}~
%\mathcal{E}_{n-1} \cap \{\xi_{n-1} =m \}.
%\]
{\rm (a)}  $ \mathcal{E}_n^\delta$ is $ \mathcal{F}_n$-measurable.
For each $N
> 0$ there is $M$ such that for $m \geq M$ we
have
\begin{equation} \label{first1}
{\mathrm{P}}(\mathcal{E}_{n}^\delta| \mathcal{F}_{n-1}) \geq 1 -
2^{-Nm} ~~{\it almost}~{\it surely}~{\it on}~ \{\xi_{n-1} =m \}
\cap \mathcal{E}_{n-1}^\delta.
\end{equation}
{\rm (b)} For each $N > 0$
 there  exist $M$ and $0 < c < 1$, such that for $m \geq M$ we
have
\[
{\mathrm{P}}(t_n - t_{n-1} > 2^{3m} k| \mathcal{F}_{n-1}) \leq c^k
+2^{-Nm} ,~~{k \geq 1},~ {\it almost}~{\it surely}~{\it on}~
\{\xi_{n-1} =m \}\cap \mathcal{E}_{n-1}^\delta.
\]
{\rm (c)}
 There  exist $M$ and $0 < c < 1$, such that for $m \geq M$
and $n \geq 2$ we have
\[
{\mathrm{P}}(t_n - t_{n-1} < 2^{3m}| \mathcal{F}_{n-1}) \leq
c~~{\it almost}~{\it surely}~{\it on}~ \{\xi_{n-1} =m \}\cap
\mathcal{E}_{n-1}^\delta.
\]
{\rm (d)} There is $p
> 1/2$ such that
for each $\varepsilon >0$ there exists $M$,  such that for $m \geq
M$ and $n \geq 2$ we have
\[
 |{\mathrm{P}}(\xi_n = \xi_{n-1}+1|
\mathcal{F}_{n-1} ) - p | \leq \varepsilon~~{\it almost}~{\it
surely}~{\it on}~ \{\xi_{n-1} =m \}\cap \mathcal{E}_{n-1}^\delta.
\]
\end{lemma}
\proof (a) The fact that $ \mathcal{E}_n^\delta$ is $
\mathcal{F}_n$-measurable follows from the definition of $
\mathcal{E}_n^\delta$ and $ \mathcal{F}_n$.

Next, observe that if $\mathcal{E}_{n-1}^\delta$ happens, then to
ensure that $\mathcal{E}_{n}^\delta$ happens we need to exclude
three events: $\{t_n=\infty\},$ $\{\tau_{k+1}-\tau_k>
|v_k|^{-1+\delta}$ for some $t_{n-1}\leq \tau_k\leq t_n\}$ and
$$\{\sup_{\tau_k\leq t\leq \tau_{k+1}} |\dot{Y}-v_k| >
|v_k|^{-1+\delta} \text{ for some } t_{n-1}\leq \tau_k\leq t_n \}
. $$

%To prove (\ref{first1}), we write
%\[
%{\mathrm{P}}(\mathcal{E}_{n}^\delta| \mathcal{F}_{n-1}) = \sum_{k
%=0}^\infty {\mathrm{P}}(\mathcal{E}_{n}^\delta \cap \{t_{n-1} =
%\tau_k\}| \mathcal{F}_{n-1}) = \sum_{k =0}^\infty
%{\mathrm{P}}(\mathcal{E}_{n}^\delta | \mathcal{G}_{k}) \chi_{
%\{t_{n-1} = \tau_k\} }.
%\]
Therefore \eqref{first1} follows from
Corollary~\ref{fie} with $b=4$ and $r = 2^{m-1}.$
%the expression in the right-hand side is estimated from below by $
%1 - 2^{-Nm}$ on  $ \{\xi_{n-1} =m \} \cap
%\mathcal{E}_{n-1}^\delta$ for all sufficiently large~$m$.

(b) The statement follows from (\ref{iiu}) once we notice that $
|\dot{Y}(t_{n-1})| \in (2^{m-1}, 2^{m+1})$ on $\{\xi_{n-1} =m
\}\cap \mathcal{E}_{n-1}^\delta$ if $m$ is sufficiently large.

(c) This follows  from Lemma~\ref{ole} once we take into account
that, by the definition of  $ \mathcal{E}_{n-1}^\delta$, for $n
\geq 2$ and all sufficiently large $m$ we have
\begin{equation} \label{ebel}
||\dot{Y}(t_{n-1})| - 2^m| \leq 1~~{\rm on}~ \{\xi_{n-1} =m \}\cap
\mathcal{E}_{n-1}^\delta.
\end{equation}

(d) Consider the limiting process $\overline{V}(t)$ with $
|\overline{V}(0)| = 1$.  Let $p$ be the probability that the
process $ |\overline{V}(t)|$ reaches $2$ before reaching $1/2$.
Notice that $p > 1/2$.  Therefore, the statement follows from
Lemma~\ref{ole} and \eqref{ebel}.
 \qed
\begin{lemma}
\label{ffi}
For  $\delta > 0$ we have
\[
\lim_{|v_0| \rightarrow \infty} {\mathrm{P}}\left( (|v_0| +
t^{1/3})^{1 - \delta} \leq | \dot{Y}(t)| \leq (|v_0| +t^{1/3})^{1
+ \delta} ~~{\it for}~{\it all}~t \geq 0 \right)  = 1.
\]
\end{lemma}
\proof
%First, let us prove a similar statement for the discrete
%process. Namely, we claim that
%\begin{equation} \label{discrete1}
%\lim_{|v_0| \rightarrow \infty} {\mathrm{P}}\left( (2^{\xi_0} +
%t_n^{\frac{1}{3}})^{1 - \delta} \leq 2^{\xi_n} \leq (2^{\xi_0}
%+t_n^{\frac{1}{3}})^{1 + \delta} ~~{\rm for}~{\rm all}~n \geq 0
%\right) = 1.
%\end{equation}
Let  $ \overline{\varepsilon}$ be fixed and $\varepsilon$ be a
positive constant, to be specified later. Let $A_{v_0}$ be the
following event
\[
A_{v_0} =(\bigcap_{n=0}^\infty \mathcal{E}_{n}^\varepsilon ) \cap
\{ |\xi_n - p n - \xi_0| \leq \varepsilon (n + \xi_0)~~{\rm
for}~{\rm all}~n \}
\]
($p$ is the constant from Lemma \ref{qqaax}(d)).

From parts (a) and (d) of Lemma~\ref{qqaax} it easily follows that
we can take a large enough $M$ such that
\begin{equation} \label{event}
{\mathrm{P}} \left( A_{v_0}  \right) \geq 1 -
\overline{\varepsilon}/3
\end{equation}
if $v_0$ is such that $\xi_0 > M$. By part (b) of
Lemma~\ref{qqaax},
\[
{\mathrm{P}} \left(A_{v_0}  \cap \{ t_n - t_{n-1} \geq k(n) 2^{3(
p(n-1) + \xi_0 +  \varepsilon (n-1 + \xi_0))} \} \right) \leq
c^{k(n)} + 2^{-N(p(n-1) + \xi_0 -  \varepsilon (n-1 + \xi_0))}
\]
 for
each $n$, where $0 < c < 1$.
 Take $k(n) =
2^{\varepsilon (n + \xi_0)}$. Let
\[
B_{v_0} = \{ t_n - t_{n-1} \geq 2^{\varepsilon (n + \xi_0)}  2^{3(
p(n-1) + \xi_0 +  \varepsilon (n-1 + \xi_0))}~{\rm for}~{\rm
some}~n \}.
\]
Then
\[
{\mathrm{P}}(A_{v_0} \cap B_{v_0}) \leq \sum_{n=1}^\infty \left(
c^{ 2^{\varepsilon (n + \xi_0)}} + 2^{-N(p(n-1) + \xi_0 -
\varepsilon (n-1 + \xi_0))} \right).
\]
The right-hand side of this inequality can be made smaller than $
\overline{\varepsilon}/3$ by taking sufficiently large $M$.

Notice that for each $a(n)$ and $\overline{k}(n)$
\begin{equation} \label{incljj}
 A_{v_0} \cap \{ t_n < a(n) \}
\subseteq  A_{v_0} \cap \{ t_n - t_{n-1} < a(n) \} \cap ... \cap
\{ t_{n - \overline{k}(n)} - t_{n-\overline{k}(n) -1} < a(n) \}.
\end{equation}
Let $\overline{k}(n) = \varepsilon (n + \xi_0)$ and $a(n) =
2^{3(p(n-1-\overline{k}(n)) + \xi_0 -  \varepsilon(n-1 +
\xi_0))}$. By part (c) of Lemma~\ref{qqaax}, the probability of
the event in the right-hand side of (\ref{incljj}) is estimated
from above by $c^{\overline{k}(n) + 1}$, where $0 < c < 1$. Let
\[
C_{v_0} = \{ t_n < 2^{3(p(n-1-\varepsilon (n + \xi_0)) + \xi_0 -
\varepsilon(n-1 + \xi_0))}~{\rm for}~{\rm some}~n \}.
\]
Then
\[
{\mathrm{P}}(A_{v_0} \cap C_{v_0}) \leq \sum_{n=1}^\infty
 c^{\varepsilon (n + \xi_0) + 1}.
 \]
The right-hand side of this inequality can be made smaller than $
\overline{\varepsilon}/3$ by taking sufficiently large $M$. We
have thus obtained that
\[
{\mathrm{P}}(A_{v_0} \setminus ( B_{v_0} \cup C_{v_0}) ) \geq 1 -
\overline{\varepsilon}.
\]
On the event $ A_{v_0} \setminus ( B_{v_0} \cup C_{v_0}) $ we have
\[
|\xi_n - p n - \xi_0| \leq  \varepsilon (n +  \xi_0)~~~{\rm
for}~{\rm all}~n;
\]
\[
t_n - t_{n-1} \leq  2^{3( p(n-1) + \xi_0 +  \varepsilon (n-1 +
\xi_0)) + \varepsilon (n + \xi_0)}~~ {\rm for}~{\rm all}~n;
\]
\[
t_n \geq 2^{3(p(n-1-\varepsilon (n + \xi_0)) + \xi_0 -
\varepsilon(n-1 + \xi_0))} ~~ {\rm for}~{\rm all}~n.
\]
Since $\varepsilon$ can be taken arbitrarily small, these three
inequalities imply that for each $\delta > 0$
\[
(2^{\xi_0} + t_n^{\frac{1}{3}})^{1 - \delta} \leq 2^{\xi_n} \leq
(2^{\xi_0} +t_n^{\frac{1}{3}})^{1 + \delta} ~~{\rm for}~{\rm
all}~n \geq 0
\]
on $ A_{v_0} \setminus ( B_{v_0} \cup C_{v_0}) $, provided that
$M$ is sufficiently large. This implies the statement of the lemma
since $ 2^{\xi_n -2 } \leq | \dot{Y}(t)| \leq 2^{\xi_n+2}$ for $
t_n \leq t \leq t_{n+1}$ on $ A_{v_0} \setminus ( B_{v_0} \cup
C_{v_0}) $ due to (\ref{ebel}). \qed
\begin{corollary}
For $\delta > 0$ we have
\begin{equation} \label{fr1}
\lim_{|v_0| \rightarrow \infty}{\mathrm{P}}\left(|v_n|^{-1} \leq
\tau_{n +1} - \tau_n \leq |v_n|^{-1 + \delta}~ {\it for}~{\it
all}~n \geq 0 \right) = 1,
\end{equation}
\begin{equation} \label{fr2}
\lim_{|v_0| \rightarrow \infty} {\mathrm{P}}\left( (n |v_0|^{-1} +
n^{3/4}) ^{1 - \delta} \leq \tau_n \leq   (n |v_0|^{-1} + n^{3/4})
^{1 + \delta} ~~{\it for}~{\it all}~n \geq 0\right) =1.
\end{equation}
\end{corollary}
\proof The first statement easily follows from (\ref{bb1}). Then
(\ref{fr2}) follows from Lemma~\ref{ffi}  and (\ref{fr1})
(considering the cases $\tau_n < |v_0|^3$ and $\tau_n \geq |v_0|^3$
separately). \qed

Let $D^\delta_{v_0}$ be the following event
\[
D^\delta_{v_0} = \{ (|v_0| + t^{1/3})^{1 - \delta} \leq |
\dot{Y}(t)| \leq (|v_0| +t^{1/3})^{1 + \delta} ~~{\it for}~{\it
all}~t \geq 0 \} \cap
\]
\[
\cap \{ |v_n|^{-1} \leq \tau_{n +1} - \tau_n \leq |v_n|^{-1 +
\delta}~ {\it for}~{\it all}~n \geq 0 \} \cap
\]
\[
\cap \{  (n |v_0|^{-1} + n^{3/4}) ^{1 - \delta} \leq \tau_n \leq
(n |v_0|^{-1} + n^{3/4}) ^{1 + \delta} ~~{\it for}~{\it all}~n
\geq 0 \}.
\]
As we saw above,
\[
\lim_{|v_0| \rightarrow \infty}
{\mathrm{P}}\left(D^\delta_{v_0}\right)=1.
\]
%We shall also need
The next result provides a more precise information about the growth of
$\dot{Y}(t)$ but only for a fixed value of $t.$
%which can be proved similarly to Lemma~\ref{ffi}.

\begin{lemma} \label{last}
We have the following limit
\[
\lim_{|v_0| \rightarrow \infty} \liminf_{a \rightarrow \infty}
\liminf_{t \rightarrow \infty} {\mathrm{P}}\left( \frac{1}{a}
t^{1/3} \leq | \dot{Y}(t)| \leq a t^{1/3}  \right)  = 1.
\]
\end{lemma}

\begin{proof}
%The proof of this lemma is very similar to the proof of Lemma \ref{ffi}
%so we just sketh the argument.

First let us estimate the probability that $|\dot{Y}(t)|$ is too
large. To this end, let $0 \leq \delta \leq 1/4$,  $m$ be the
largest integer such that $2^m\leq {a t^{1/3}}/{4}$, and $n^*$ be
the first time when $\xi_n=m.$ Then $(t_{n^*+1}-t_{n^*})/ta^3$ is
tight by Lemma \ref{ole}, and therefore $\mathrm{P}(D^\delta_{v_0}
\cap \{t_{n^*+1}-t_{n^*}\leq t \} )\to 0$ as $a\to\infty$
uniformly in $t \geq 1$. Since $\max_{t\leq t_{n^*+1}}
|\dot{Y}(t)|\leq at^{1/3}$ on $ D^\delta_{v_0}$ for large $|v_0|$,
we see that $$ \mathrm{P}(|\dot{Y}(t)|\geq a t^{1/3}) $$ can be
made as small as we wish by choosing $a$ and $|v_0|$ large.

Proving that $\dot{Y}(t)$ is unlikely to be small requires a more
sophisticated argument. We show that $\dot{Y}$ can not be small on
whole $[0, t]$ since this would require that $\dot{Y}$ spends a
lot of time near $2^p$ for some small $p$ contradicting the
transience of $\xi_n.$ On the other hand Lemma \ref{ole} allows us
to rule out the possibility that $\dot{Y}(t)$ is small while
$\dot{Y}(s)$ is large for some $s\in [0,t]$ since $0$ is an
inaccessible point for $\overline{V}(t).$

Let us give the precise argument. Let $n_*=n_*(b,t)$ be the first time when
$|\dot{Y}(t_n)|\geq b t^{1/3}.$
To estimate the probability that $|\dot{Y}(t)|$ is too small, it
is enough to show that
\begin{equation}
\label{SpringUp}
\lim_{|v_0| \rightarrow \infty} \liminf_{b
\rightarrow 0} \liminf_{t \rightarrow \infty} {\mathrm{P}} \left(
t_{n_*}\leq t \right)  = 1
\end{equation}
since, by Lemma \ref{ole}, for fixed $b$ $$
\mathrm{P}\left(\min_{s\in[t_{n_*},t_{n_*}+t]} |\dot{Y}(s)|\leq
\frac{t^{1/3}}{a}\right) $$ can be made as small as we wish by
taking $a$ large  uniformly in $t \geq 1$.

Let $$ T(p)=\sum_{n=1}^\infty \left(t_{n+1}-t_n\right)
\chi_{\{\xi_n=p\}}. $$ Let $p_*=\log_2 (b t^{1/3})+2$. Observe
that on $D_{v_0}^\delta$ we have $$ (1-\delta) \log_2 |v_0| \leq
\xi_n\leq p_* \text{ for } n\leq n_*. $$ Let $$ F_{v_0,
K}=\{T(p)\leq K 2^{3p} \times 2^{p_*-p}~~ \text{for all $p$ such
that}~~ (1-\delta) \log_2 |v_0| \leq p \leq p_*\} . $$

%Since $\xi_n\leq \log_2 (b t^{1/3}) $ for
%$n\leq n_*$, to establish \eqref{SpringUp} it is enough to show
We claim that for each $N > 0$ there are $c_0 >0$, $c<1,$ $p_0>0$
such that for $p\geq p_0$, all $k$ and $v_0$ we have
\begin{equation}
\label{Stick} \mathrm{P}\left(T(p)\geq 2^{3p} k\right)\leq c_0
\sqrt{k} \left(c^{\sqrt{k}}+ 2^{-N p}\right).
\end{equation}
Note that \eqref{Stick} implies that for each $\varepsilon>0$
there exist constants $K, r$ such that for $|v_0|\geq r$ we have $
P(F_{v_0, K})\geq 1-\varepsilon$.   Also note that \eqref{Stick}
implies \eqref{SpringUp} since on $D_{v_0}^\delta\bigcap F_{v_0,
K}$ we have for all sufficiently large $|v_0|$ $$ t_{n_*}\leq
\sum_{p=(1-\delta) \log_2 |v_0|}^{p_*} K 2^{3p} \times
2^{p_*-p}.$$

 To establish \eqref{Stick} we note that $$ \mathrm{P}
\left(\#(n: \xi_n=p)\geq \sqrt{k} \right)\leq c^{\sqrt{k}} $$
since every time $\xi_n$ visits $p$ it has a positive probability
of never returning there. On the other hand by Lemma
\ref{qqaax}(b) $$ \mathrm{P} \left( \max_{n: \xi_n=p}
(t_{n+1}-t_n) \geq 2^{3p} \sqrt{k}|\#(n:\xi_n=p)<\sqrt{k} \right)\leq \sqrt{k}
\left(c^{\sqrt{k}}+ 2^{-N p}\right) $$ so \eqref{Stick} follows.
\end{proof}

\subsection{Probability of a Near Self-Intersection for $Y(t)$}
\label{SSSI}
In this section we prove that if $|v_0|$ is large, then with high
probability the `tail' of the the trajectory $Y(t)$ (the part of
the trajectory corresponding to $t \geq \tau_n$) leaves a
neighborhood of $Y_n$ and then never comes close to the part of
the trajectory corresponding to $t \leq \tau_n$. This allows us to
conclude that switching to a new version of the force field at
each of the times $\tau_n$ does not have a major effect on the
distribution of the solution, that is the distributions of $X(t)$
and $Y(t)$ are the same if we throw out events of small measure
from their respective probability spaces.

Let $\gamma_n$, $n \geq 1$,  be the trajectory of the process
$Y(t)$ between times $\tau_{n-1}$ and $\tau_n$, that is
\[
\gamma_n = \{ Y(t), \tau_{n-1} \leq t \leq \tau_n \}.
\]
Let $\Gamma_n$ be the trajectory of the process after time
$\tau_n$, that is
\[
\Gamma_n = \{ Y(t), \tau_{n} \leq t < \infty \}.
\]
Let $\gamma_n^{2R}$ be the $2R$-neighborhood of $\gamma_n$ and
$\Gamma_n^R$ the $R$-neighborhood of $\Gamma_n$. We shall prove
the following lemma.
\begin{lemma} \label{nntt}
There exists
 $0 < \delta <1$  such that
\begin{equation} \label{ede}
{\mathrm{P}}\left( D^\delta_{v_0} \cap \gamma^{2R}_n \cap
\Gamma^R_{n+1} \neq \emptyset  \right) \leq (|v_0| + n^{1/4})^{-4d
+ 12 - \delta}
\end{equation}
for all sufficiently large $|v_0|$ and all $n \geq 1$.
\end{lemma}
Before we prove Lemma~\ref{nntt}, let us make several remarks
which will, in particular, allow us to deduce parts  (a), (b) and
(c) of Lemma~\ref{lh1} from Lemma~\ref{nntt}. For $x, v \in
\mathbb{R}^d$, let $K^+(x, v)$ and $K^-(x, v)$ be the cones
\[
K^+(x,v)  = \{y \in \mathbb{R}^d: (y-x, v) \geq \frac{3}{4}|y-x|
|v| \},
\]
\[
K^-(x,v)  = \{y \in \mathbb{R}^d: (y-x, -v) \geq \frac{3}{4}|y-x|
|v| \}.
\]
From the definition of $D^\delta_{v_0}$ it easily follows that
\[
{\mathrm{P}}\left( D^\delta_{v_0} \cap \bigcup_n (\{\gamma_n
\nsubseteq K^-(Y_n , v_n) \} \cup \{  \gamma_{n+1} \nsubseteq
K^+(Y_n , v_n) \})\right) \leq |v_0|^{-N}
\]
 if $|v_0|$ is sufficiently large. This
 implies that for each $0 < \delta < 1$
\begin{equation} \label{ede00}
{\mathrm{P}}\left( D^\delta_{v_0} \cap  \bigcup_n  \{\gamma^{2R}_n
\cap \gamma^R_{n+1} \nsubseteq B_{2R}(Y_n) \} \right) \leq
|v_0|^{-N}
\end{equation}
for all sufficiently large $|v_0|$. Take  $0 < \delta < 1$ such
that (\ref{ede})  holds. Let
\begin{equation} \label{omga}
{\Omega}_{v_0} = D^\delta_{v_0} \cap \{ \gamma^{2R}_n \cap
\Gamma^R_{n+1} = \emptyset ~~{\rm for}~{\rm all}~n  \} \cap \{
\gamma^{2R}_n \cap \gamma^R_{n+1} \subseteq B_{2R}(Y_n)~~{\rm
for}~{\rm all}~n  \}.
\end{equation}
In order to see that  parts (a) and (b) of Lemma~\ref{lh1} hold,
it remains to note that
\begin{equation}
\label{FinNSI}
 \lim_{|v_0| \rightarrow \infty} \sum_{n =1}^\infty
(|v_0| + n^{1/4})^{-4d + 12 - \delta} = 0
\end{equation}
if $d \geq 4$.  From the definition of $ D^\delta_{v_0}$ it
immediately follows that $\lim_{t \rightarrow \infty} |\dot{Y}(t)|
= \infty$ on ${\Omega}_{v_0}$. Furthermore, the trajectory $Y(t)$
cannot have limit points in $\mathbb{R}^d,$
%visit a fixed compact set for arbitrarily large values of $t$,
as follows from the definition of ${\Omega}_{v_0}$.
Therefore, $\lim_{t \rightarrow \infty} |{Y}(t)| = \infty$, which
proves part (c) of Lemma~\ref{lh1}.

 Let us now return to Lemma \ref{nntt}.
 From the definition of $D^\delta_{v_0}$ it follows that if $\delta' >
0$, then $\gamma_n \subseteq B(Y(\tau_n), v_n^{\delta'}) $
 for all $n \geq
1$  if $\delta
> 0$ is sufficiently small and  $|v_0|$ is sufficiently large.
Let us represent $ \Gamma^R_{n+1} $ as follows
\[
\Gamma^R_{n+1}  = {\overline{\Gamma}}^{R}_{n+1}(\delta) \cup
\overline{\overline{\Gamma}}^R_{n+1} (\delta),
\]
where ${\overline{\Gamma}}^{R}_{n+1}(\delta)$ is the
$R$-neighborhood of ${\overline{\Gamma}}_{n+1}(\delta) = \{ Y(t),
\tau_{n+1} \leq t \leq \tau_n + |v_n|^{3 - \delta} \}$ and $
\overline{\overline{\Gamma}}^R_{n+1} (\delta)$ is the
$R$-neighborhood of $\overline{\overline{\Gamma}}_{n+1} (\delta) =
\{ Y(t),\tau_n + |v_n|^{3 - \delta} \leq t \leq \infty \}$.

Recall that the constant $l$ from the definition of the stopping
time $\tau_n$ is equal to $4R+1$. Since Lemma~\ref{ccf} is
obviously also applicable to the process $Z(t)$,
\[
\mathrm{P}(D^\delta_{v_0} \cap \{ {\rm dist} (K^-(Y_n, v_n),
{\overline{\Gamma}}_{n+1}(\delta) ) \leq 3 R\} ) \leq (|v_0| +
n^{1/4})^{-4d + 12 - \delta}
\]
if $\delta > 0$ is sufficiently small and $|v_0|$ is sufficiently
large. This implies (\ref{ede}) with
${\overline{\Gamma}}^{R}_{n+1}(\delta) $ instead of $
{\Gamma}^{R}_{n+1}$. Thus, Lemma~\ref{nntt} will follow if we
prove that
\begin{equation} \label{ede44}
{\mathrm{P}}\left( D^\delta_{v_0} \cap \gamma^{2R}_n \cap
\overline{\overline{\Gamma}}^R_{n+1} (\delta) \neq \emptyset
\right) \leq (|v_0| + n^{1/4})^{-4d + 12 - \delta}
\end{equation}
\begin{lemma} \label{le1} There exist $0 < \varepsilon < 1$ and $0 < \delta_0 < 1$ such that
for  each $0 < \delta, \delta' < \delta_0$ and $R'$ the inequality
\begin{equation} \label{oop}
{\mathrm{P}}\left(D^\delta_{v_0} \cap \{ |Y(\tau_n + |v_n|^{3 -
\delta} + t) - x| \leq R' \} | \mathcal{G}_{n} \right) \leq
(|v_n|^{3 - \delta} + t)^{-\frac{4}{3}(d-2) -\varepsilon }
\end{equation}
 holds for
all sufficiently large $|v_0|$ uniformly in $n \geq 0$, $x \in
B(Y(\tau_n), v_n^{\delta'})$ and $t \geq 0$.
\end{lemma}
\proof In view of Lemma \ref{ccf} we can assume that
$t>|v_n|^{3-\delta}.$ Denote $\Tt=\tau_n + |v_n|^{3 - \delta} +
t.$ Let us first explain the proof of a weaker bound: for each
$\varepsilon > 0$ we have
\begin{equation}
\label{d5} {\mathrm{P}}\left(D^\delta_{v_0} \cap \{ |Y(\Tt) - x|
\leq R' \} | \mathcal{G}_{n} \right) \leq  t^{-\frac{4}{3}(d-2) +
\varepsilon }.
\end{equation}
This suffices for $d>4$ (see the proof of Lemma~\ref{nntt}). Then
we explain how to improve this estimate to get \eqref{oop}. The
proof of \eqref{d5} consists of two steps.

(I)  Fix $\varepsilon_1 > 0$. We show that if the intersection
does take place and $D^\delta_{v_0}$ takes place then with high
probability there exists a number $k$ such that
$\tau_n+t^{1-\varepsilon_1} \leq \tau_{n+k}\leq \Tt$ and the
following conditions are satisfied.

(A) $|Y(\tau_{n+k})-x|\geq t^{4/3-\varepsilon} ,$

(B) $\frac{\pi}{4} \leq \angle((Y(\tau_{n+k})-x), v_{n+k})\leq
\frac{3\pi}{4}. $

(II)  By step (I) it suffices to show that
\begin{equation}
\label{SmDen} \mathrm{P}\left(Y(\Tt)\in B(x, R')\text{ and (A) and
(B) hold} \right)\leq \Const t^{-\frac{4}{3}(d-2) + \varepsilon }
\end{equation}
 To prove \eqref{SmDen}, denote $r=|Y(\tau_{n+k})-x|$, let $\Pi$ be the plane passing through
$x$ orthogonal to $v_0$ and let $Pr$ denote the projection to
$\Pi.$ We can find a set $S=\{x_j\}$ of cardinality at least $c
r^{d-2}$ such that $x_1=x,$ the balls $B(x_j, R')$ are disjoint,
and for each $j$ there is an isometry $\cO_j$ leaving
$Y(\tau_{n+k})$ and $v_{n+k}$ fixed and such that
$\cO_j(x_j)=x_1.$ By the rotation invariance,
\begin{equation} \label{feight}
\mathrm{P}\left(Pr(Y(\Tt))\in B(x_1, R')\right)\leq
\frac{1}{\text{Card}(S)}
\end{equation}
 proving \eqref{SmDen}.

Thus to complete the proof of \eqref{d5} it remains to justify
step I. Observe that on $D_{v_0}^\delta$ we have $$
|Y(\tau_n+t^{1-\varepsilon_1})-Y(\tau_n)|\leq \Const
t^{(\frac{4}{3}-\eps_1)(1+\delta)} $$ $$
|Y(\Tt)-Y(\Tt-t^{1-\varepsilon_1})|\leq \Const
t^{(\frac{4}{3}-\eps_1)(1+\delta)} $$ On the other hand the
inequality  $\angle(\dot{Y}(s), (Y(s)-x))\leq \frac{\pi}{3}$ for
all $s\in [\tau_n+t^{1-\eps_1}, \Tt-t^{1-\eps_1}]$ would imply $$
\left||Y(\Tt-t^{1-\eps_1})-x|-|Y(\tau_n+t^{1-\eps_1})-x|\right|\geq
\Const t^{4/3(1-\delta)} $$ making intersection impossible if
$\eps_1>3\delta.$ Thus there exists $t_1\in [\tau_n+t^{1-\eps_1},
\Tt-t^{1-\eps_1}]$ such that $\angle(\dot{Y}(t_1),
(Y(t_1)-x))=\frac{\pi}{3}.$ Next with high probability the angle
changes less than $\frac{\pi}{12}$ on $[t_1,
t_1+t^{(1-\eps_1)(3-\delta)}].$ Thus the motion on this interval
is well approximated by a straight line and consequently there is
$t_2\in [t_1, t_1+t^{(1-\eps_1)(3-\delta)}]$ such that $$
|Y(t_2)-x|\geq \Const t^{(1-\eps_1)(3-\delta)} t^{1/3-\delta} . $$
Taking $k$ to be the first number such that $\tau_{n+k}>t_2$
establishes our claim.

Now let us now indicate how to prove the lemma in full generality.
We need to prove \eqref{SmDen} with $-\varepsilon$ instead of
$\varepsilon$ in the right-hand side.  In the arguments leading to
\eqref{SmDen} we only used the projection on the plane orthogonal
to $v_{n+k}$. Now we consider the projection of the process onto
the $v_{n+k}$ direction. During the time interval between
$\widetilde{t}- t^{1/10}$ and $\widetilde{t}$ the projection of
$\dot{Y}(s)$ can be well-approximated by a martingale, and as such
by a time-changed Brownian motion. The time-change is almost
linear on this small time interval, and thus the projection of
$Y(s)$ is approximated well by the integral of the Brownian
motion. This allows us to gain an extra factor of $t^{-2
\varepsilon}$. Observe that the derivation of \eqref{feight} only
involved rotation-invariance, and thus \eqref{feight} remains
valid if we replace the probability in the left-hand side by
conditional probability with the condition which involves the
projection of the process on the direction of $v_{n+k}$.
 \qed
\\
\\
{\it Proof of Lemma~\ref{nntt}.}
 Let
\[
 s^n_k(\delta) = \tau_{n-1} + k (|v_{n}| + n^{1/4})^{- 3 \delta},~
 k = 0,...,[(\tau_n - \tau_{n-1})(|v_{n}| + n^{1/4})^{ 3
\delta}].
\]
As follows from the definition of $D^\delta_{v_0}$, for each $R'$
these points form an $R'$-net in $\gamma_n$ if $|v_0|$ is
sufficiently large. By applying (\ref{oop}) to $x^n_k(\delta) =
Y(s^n_k(\delta))$, we obtain that
\[
{\mathrm{P}}\left(D^\delta_{v_0} \cap \{ {\rm dist}(Y(\tau_n +
|v_n|^{3 - \delta} + t), \gamma_n) \leq R' \} | \mathcal{G}_{n}
\right) \leq
\]
\[
\leq (|v_n|^{3 - \delta} + t)^{-\frac{4}{3}(d-2) -\varepsilon}
(|v_{n}| + n^{1/4})^{ 3 \delta}
\]
 holds for
all sufficiently large $|v_0|$ uniformly in $n \geq 0$ and $t \geq
0$.  Since $ | \dot{Y}(t)| \leq (|v_0| +t^{1/3})^{1 + \delta}$ on
$D^\delta_{v_0}$, and $R'$ was arbitrary,
\[
{\mathrm{P}}\left( D^\delta_{v_0} \cap \gamma^{2R}_n \cap
\overline{\overline{\Gamma}}^R_{n+1} (\delta) \neq \emptyset |
\mathcal{G}_{n} \right) \leq {\mathrm{P}}\left(D^\delta_{v_0} \cap
\{ {\rm dist}(\overline{\overline{\Gamma}}_{n+1} (\delta),
\gamma_n) \leq 3R \} | \mathcal{G}_{n} \right) \leq
\]
\begin{equation} \label{integral}
\leq \int_0^\infty (|v_n|^{3 - \delta} + t)^{-\frac{4}{3}(d-2)
-\varepsilon} (|v_{n}| + n^{1/4})^{ 3 \delta} (|v_0| +(\tau_n +
|v_n|^{3 - \delta} +t )^{1/3})^{1 + \delta} d t
\end{equation}
holds for all sufficiently large $|v_0|$ uniformly in $n \geq 0$.
It follows from the definition of $D^\delta_{v_0} $ that
\[
|v_n| \leq ( |v_0| + n^{1/4}) ^{1 + 3\delta}
\]
on  $D^\delta_{v_0} $  for all sufficiently large $|v_0|$. Recall
that
\[
\tau_n \leq (n |v_0|^{-1} + n^{3/4}) ^{1 + \delta}
\]
on  $D^\delta_{v_0} $  for all sufficiently large $|v_0|$. Since
$\varepsilon$ is fixed,  these estimates imply
that right-hand side of (\ref{integral}) can be made smaller than
the right-hand side of (\ref{ede44}) by taking a sufficiently
small $\delta$. \qed

\section{The Convergence in Distribution}
\label{YY}

Here we prove Lemma \ref{lh1}(d). Recall that ${\Omega}_{v_0}$ is
given by (\ref{omga}).
%In what follows the probability space is
%assumed to be
% ${\Omega}_{v_0}$ with the normalized measure ${{\mathrm{P}}}_{v_0}$.

 For
fixed $v_0$,  let us prove that the family of processes
$\dot{Y}(c^3t)/c$ is tight, when restricted to the event
${\Omega}_{v_0}$. By the Arzela-Askoli Theorem, it is sufficient
to show that for each $T, \varepsilon, \eta > 0$ there are $c_0$
and $\varkappa > 0$ such that
\begin{equation}
\label{NUC} {\mathrm{P}} \left( {\Omega}_{v_0} \cap \{\sup_{0 \leq
s \leq t \leq T, t -s \leq \varkappa} |\dot{Y}(c^3t)/c  -
\dot{Y}(c^3s)/c|
> \varepsilon  \} \right) < \eta
\end{equation}
for $c \geq c_0$.

%In order for the process $|\dot{Y}(c^3t)|$ to
%attain the value $A  c$, it must first attain the value $A c /2$.
%The probability that the transition from $A c /2$ to $A c$ for the
%process $|\dot{Y}(t)|$ takes less time than $T c^3$ can be made
%arbitrarily small for all sufficiently large $c$ and all
%sufficiently large $A$ due to Lemma~\ref{ole}.

Let $T, \varepsilon, \eta > 0$ be fixed.
% We can assume that  $|\dot{Y}(c^3t)|$  attains the value $A  c$ with
%probability less than $\eta/2$.
Let $n_*=n_*(\kappa,c)$ be the first time  when
$|\dot{Y}(\tau_n)|\geq \kappa c.$ Take $\kappa<{\eps}/{4}$. Define
$U_{\kappa, c}(t)=\dot{Y}(\tau_{n_*}+c^3 t)/c.$ By Lemma
\ref{ole}, there is $\varkappa>0$ such that
$$\mathrm{P}\left({\Omega}_{v_0} \cap \{\sup_{0 \leq s \leq t \leq
T, t -s \leq \varkappa}  |U_{\kappa, c}(t)-U_{\kappa, c}(s)| >
\frac{\eps}{2} \} \right)<\eta $$ for large $c.$ Now \eqref{NUC}
follows easily.

From Lemma \ref{last}, the definition of $D^\delta_{v_0}$, and the
tightness established above it follows that for each $T,
\varepsilon, \eta > 0$ there is $\kappa>0$ such that
\[
 \mathrm{P}\left({\Omega}_{v_0} \cap \{ \sup_{t \in [0,T]}
\left|U_{\kappa, c}(t)-\frac{\dot{Y}(tc^3)}{c}\right|\geq
\varepsilon \} \right) =
\]
\[
\mathrm{P}\left({\Omega}_{v_0} \cap \{ \sup_{t \in [0,T]}
\left|\frac{\dot{Y}(\tau_{n_*}+
tc^3)}{c}-\frac{\dot{Y}(tc^3)}{c}\right|\geq \varepsilon \}
\right) <\eta
\]
 for all sufficiently large $c$.
Likewise, if $\brtau_\kappa$ is the first time when
$|\overline{V}(\brtau)|=\kappa$, define
$\overline{U}_\kappa(t)=\overline{V}(\brtau_\kappa+t).$ Then for
each $T, \varepsilon, \eta > 0$ there is $\kappa>0$ such that $$
\mathrm{P}\left(\sup_{t \in [0,T]}
\left|\overline{U}_{\kappa}(t)-\overline{V}(t)\right|\geq
\varepsilon\right)<\eta. $$ Finally, from Lemma \ref{ole} and the
definition of $ {\Omega}_{v_0}$ it follows that the distribution
of
 $U_{\kappa, c}$, considered over the space ${\Omega}_{v_0}$ with
 the normalized measure,
  is close to the distribution of $\overline{U}_\kappa$  if
$c$ is large enough. This completes the proof of Lemma \ref{lh1}.

\section{Appendix}

Here we sketch the proof of Lemma~\ref{nlnl}. Note that it is
sufficient to prove (\ref{star}), since (\ref{bb5}) follows from
(\ref{star}) in the same way as Corollary~\ref{fttt} follows from
Lemma~\ref{ftt01}. We use the same notations as in the proof of
Lemma~\ref{ftt01}. It is clear that $ \eta_{1} \leq T_0 =
|v_0|^\alpha$ with high probability. Therefore, due to
\eqref{teyy} and \eqref{bb3} it suffices to show that $$
\mathrm{E}\left(\left|\int_{\tau_1}^{\eta_1} \tilde F_0(z(s_0)) ds
\right| \chi_{ \{\max(\tau_1, \eta_1) \leq T_0 \} } \right)\leq
|v_0|^{-3+\delta}$$ for all sufficiently large $|v_0|$. Since
$\tilde F$ is a Poisson field, the problem is reduced to showing
that
% and the proximity of $y(t)$ and $z_0(t)$ (formula (\ref{gr1aa})),
% (\ref{star}) will follow if only we show that
for each $\delta > 0$ one can choose $\alpha > 0$ such that
\[
\mathrm{E}\left(\left|\tau_1 - \eta_1\right|
\chi_{ \{\max(\tau_1, \eta_1) \leq T_0 \} } \right)
\leq |v_0|^{-3 + \delta}
\]
for all sufficiently large $|v_0|$.
We shall only prove that
\begin{equation} \label{etau}
\mathrm{E}((\tau_1 - \eta_1)^+ \chi_{ \{\max(\tau_1, \eta_1) \leq
T_0 \} }) \leq |v_0|^{-3 + \delta}
\end{equation}
since the inequality with $\eta_1 - \tau_1$ instead of $\tau_1 -
\eta_1$ can be proved similarly.

Let $\gamma > 0$ and $0 \leq q \leq 2$. We shall specify these
constants later.  For simplicity of notation, assume that $v_0$ is
directed along the $x_1$-axis, in the positive direction.  Let
$S_{q, \gamma}$ and $S^+_{q, \gamma}$ be the following random
sets:
\[
S_{q, \gamma} = \{x \in \mathbb{R}^d: {\rm dist}(x, z_0(\eta_1))
\geq 2R,~ v_0 \eta_1 - 2R \leq x_1 \leq v_0 \eta_1,
\]
\[
2R - |v_0|^{-q + \gamma} \leq \sqrt{x_2^2+...x_d^2} \leq 2R -
|v_0|^{-q} +  |v_0|^{-2} \},
\]
\[
S^+_{q, \gamma} = \{x \in \mathbb{R}^d:  {\rm dist}(x,
z_0(\eta_1)) \geq 2R,~v_0 \eta_1 - 2R \leq x_1 \leq v_0 \eta_1,
\]
\[
2R - |v_0|^{-q} +  |v_0|^{-2} < \sqrt{x_2^2+...x_d^2} \leq 2R\}.
\]

Let $\Gamma_{q, \gamma}$ be the following random set:
\[
\Gamma_{q, \gamma} = \{x \in \mathbb{R}^d: x_1 = v_0 \eta_1 +
|v_0|^{-2 + 2 \gamma + \frac{q}{2}}, \sqrt{x_2^2+...x_d^2} \leq
|v_0|^{-2 + \gamma} \}.
\]
Let $U_{q, \gamma}$ be the following random set:
\[
U_{q, \gamma} = \{x \in \mathbb{R}^d:  {\rm dist}(x, z_0(\eta_1))
\geq 2R,~{\rm dist}(x, \Gamma_{q, \gamma} ) \leq 2R \}.
\]
 \begin{center}
   \includegraphics[angle=0, scale=1.0]{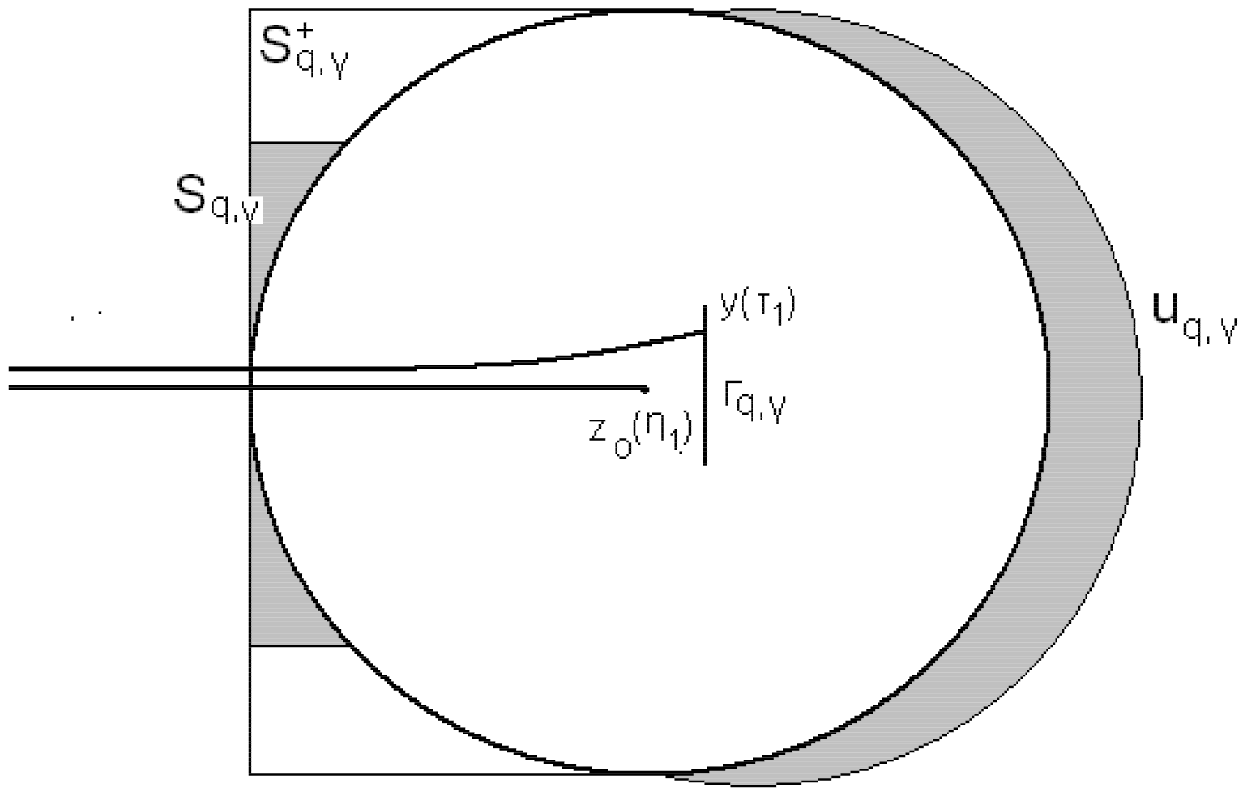}
   \end{center}
Let $ \mathcal{E}^S_{q, \gamma}$ be the event that at least one of
the points $r_1,r_2,...$ belongs to $ S_{q, \gamma}$ but none
belong to $ S^+_{q, \gamma}$. Let $ \mathcal{E}^U_{q, \gamma}$ be
the event that at least one of the points $r_1,r_2,...$ belongs to
$ U_{q, \gamma}$. Let $A$ be a point on the semi-axis $\{x \in
\mathbb{R}^d: x_1 \geq 0, x_2 = ... = x_d = 0 \}$. Note that $
\mathcal{E}^S_{q, \gamma}$ and $ \mathcal{E}^U_{q, \gamma}$ are
independent when conditioned on $\{z_0 (\eta_1) = A \}$. The
respective conditional probabilities can be estimated from above
by $|v_0|^{-q + 2 \gamma}$ and $|v_0|^{-2 + 3 \gamma +
\frac{q}{2}}$ for all sufficiently large $|v_0|$. Therefore, $
\mathrm{P}( \mathcal{E}^S_{q, \gamma} \cap \mathcal{E}^U_{q,
\gamma}) \leq |v_0|^{-2 + 5 \gamma - \frac{q}{2}}$.

Let us examine the contribution to the expectation (\ref{etau})
from the event $ \mathcal{E}^S_{q, \gamma}$. First,
\[
\mathrm{E}(\chi_{\mathcal{E}^S_{q, \gamma} \cap \mathcal{E}^U_{q,
\gamma}} (\tau_1 - \eta_1)^+ \chi_{ \{\max(\tau_1, \eta_1) \leq
T_0 \} }) \leq T_0 \mathrm{P}( \mathcal{E}^S_{q, \gamma} \cap
\mathcal{E}^U_{q, \gamma}) \leq |v_0|^{\alpha -2 + 5 \gamma -
\frac{q}{2}}.
\]
Note that the power $\alpha -2 + 5 \gamma - \frac{q}{2}$ can be
made less than $-3+ \delta$ by selecting small $\gamma$ and
$\alpha$ close to $-1$. Next, note that with high probability the
trajectory $y(t)$ reaches the set $\Gamma_{q, \gamma}$ between
times $\eta_1$ and $\eta_1 + |v_0|^{-3 + 3 \gamma + \frac{q}{2}}$
due to the proximity of $y(t)$ and $z_0(t)$. Note that the
distance between  $\Gamma_{q, \gamma}$  and $S^+_{q, \gamma}$ is
greater than $2R$.  Therefore, on $\mathcal{E}^S_{q, \gamma}
\setminus \mathcal{E}^U_{q, \gamma}$, none of the points
$r_1,r_2,...$ belongs to the $2R$-neighborhood of the point where
$y(t)$ first intersects $\Gamma_{q, \gamma}$. Therefore, for each
$N
> 0$,
\[
\mathrm{E}(\chi_{\mathcal{E}^S_{q, \gamma} \setminus
\mathcal{E}^U_{q, \gamma}} (\tau_1 - \eta_1)^+ \chi_{
\{\max(\tau_1, \eta_1) \leq T_0 \} }) \leq |v_0|^{-3 + 3 \gamma +
\frac{q}{2}} \mathrm{P}( \mathcal{E}^S_{q, \gamma} ) + |v_0|^{-N}
\leq |v_0|^{ -3 + 6 \gamma - \frac{q}{2}}.
\]
Again, the power $ -3 + 6 \gamma - \frac{q}{2}$ can be made less
than $-3+ \delta$ by selecting small $\gamma$. We have thus
obtained that
\begin{equation} \label{eachq}
\mathrm{E}(\chi_{\mathcal{E}^S_{q, \gamma} } (\tau_1 - \eta_1)^+
\chi_{ \{\max(\tau_1, \eta_1) \leq T_0 \} })  \leq |v_0|^{-3 +
\delta}.
\end{equation}
Note that for fixed $\gamma$ one can find finitely many numbers
$q_1,...,q_n \in [0,2]$ such that $ \mathrm{P}(\bigcup_{i=1}^n
\mathcal{E}^S_{q_i, \gamma}) =1$. Therefore, (\ref{eachq}) implies
(\ref{etau}).


\begin{thebibliography}{9999}

\bibitem{Bill} Billingsley P. {\it Probability and Measure}
Wiley-Interscience, 1995.

\bibitem{DD} Dolgopyat D.,
{\it Fermi Acceleration,} Cont. Math {\bf 469} (2008) 149-166.

\bibitem{DK} Dolgopyat D., Koralov L.
{\it Avergaging of Hamiltonian flows with an ergodic component,} to appear in Ann. Prob.

\bibitem{DGL} Durr D., Goldstein S., Lebowitz J., {\it Asymptotic
Motion of a Classical Particle in a Random Potential in Two
Dimensions: Landau Model}, Comm. Math. Phys. ${\mathbf{113}}$
(1987) 209-230.
\bibitem{KP} Kesten H., Papanicolaou G. {\it A Limit Theorem for
Stochastic Acceleration,} Comm. Math. Phys. $\mathbf{78}$
(1980/81) 19-63.
%\bibitem{Kh2} Khasminskii R. Z.
%{\it A limit theorem for solutions of differential equations with a
%random right hand part,}
%Teor. Verojatnost. i Primenen {\bf 11} (1966) 444--462.
\bibitem{KR} Komorowski T., Ryzhik L. {\it Diffusion in a Weakly
Random Hamiltonian Flow}, Comm. Math. Phys. $\mathbf{263}$  (2006)
277-323.
\bibitem{KR2} Komorowski T., Ryzhik L. {\it The Stochastic Acceleration
Problem in Two Dimensions}, Israel Journal of Mathematics,
$\mathbf{155}$ (2006) 157--204.
\bibitem{KS} Karatzas I., Shreve S. {\it Brownian Motion and
Stochastic Calculus}, 2nd ed., Springer 1994.
\bibitem{RY} Revuz D. and Yor M. {\it Continuous Martingales and
Brownian Motion,} 3d ed., Sringer, Berlin-Heildelberg-New York,
1998.



\end{thebibliography}
\end{document}